%% file: article-sympl.tex
\pdfoutput=1
\documentclass{amsart}

\input{packages}
\input{configuration}
\input{macros}


\title{Graph potentials and symplectic geometry of moduli spaces of vector bundles}

\author{Pieter Belmans}
\address{Department of Mathematics, Universit\'e de Luxembourg, 6, avenue de la Fonte, L-4364 Esch-sur-Alzette, Luxembourg}
\author{Sergey Galkin}
\address{PUC-Rio, Departamento de Matem\'atica, Rua Marqu\^es de S\~ao Vicente 225, G\'avea, Rio de Janeiro; HSE University, Faculty of Mathematics, Moscow}

\author{Swarnava Mukhopadhyay}
\address{School of Mathematics, Tata Institute of Fundamental Research, 1 Homi Bhabha Road, Navy Nagar, Colaba, Mumbai 400005}

\email{pieter.belmans@uni.lu,arxiv-gp-sympl@galkin.org.ru,swarnava@math.tifr.res.in}

\begin{document}

\begin{abstract}
%For any monotone Lagrangian torus $L$ on a Fano manifold $X$
%the weighted number of holomorphic Maslov index $2$ discs with boundary on $L$
%is bounded from below by an invariant $T_X$ (introduced earlier in the Galkin--Golyshev--Iritani conjecture).
%If this bound is attained the torus $L$ is said to be optimal.
%We give the first examples of Fano manifolds with multiple optimal tori.

We give the first examples of Fano manifolds with multiple optimal tori,
i.e.~we construct monotone Lagrangian tori $L$,
such that the weighted number of holomorphic Maslov index two discs with boundary on $L$
equals the upper bound given by the symplectic invariant $\limsup_n ([m_0(L)^n]_{x^0})^{1/n}$,
where $m_0(L)$ is the Floer potential.

To every trivalent graph $\gamma$ of genus $g$
we associate an optimal torus $L_\gamma$ on the celebrated symplectic Fano manifold $\cN_g$
(of complex dimension $3g-3$) % with $\mathrm{T}_{\cN_g} = 8g-8$),
given by the character variety of rank 2 on a genus $g$ surface with prescribed odd monodromy at a puncture,
We moreover show that all pairs $(\cN_g,L_\gamma)$ are pairwise non-isotopic.
In particular, we confirm a form of mirror symmetry
between the A-model of the pairs $(\cN_g,L_\gamma)$ % (and also spaces $\cN_g$ standalone)
and B-model of graph potentials,
a family of Laurent polynomials we introduced in earlier work.

A crucial input from outside of symplectic geometry
is an analysis of Manon's toric degenerations of algebro-geometric models $\odd$ for the spaces $\cN_g$,
as moduli spaces of stable rank $2$ bundles on an algebraic curve with a fixed determinant,
constructed using conformal field theory.
\end{abstract}

\maketitle

\renewcommand{\baselinestretch}{0.75}\normalsize
\tableofcontents
\renewcommand{\baselinestretch}{1.0}\normalsize

\input{sympl-intro}
\input{sympl-toric}

\input{sympl-monotone}
\input{sympl-disks}
\input{sympl-descend}

\appendix
\input{sympl-singular}

\input{sympl-lines}

\renewcommand*{\bibfont}{\small}
\nocite{Galkin-talk}
\printbibliography
\end{document}

%% file: packages.tex
\usepackage{libertine}
\usepackage[T1]{fontenc}
\usepackage[utf8]{inputenc}      % for texlive-2016 and older
\usepackage[libertine]{newtxmath}
\usepackage[scaled=0.83]{beramono}
\usepackage{eucal}
\usepackage[babel]{microtype}
\usepackage[russian,main=american]{babel}
\usepackage[main=american]{babel}
\usepackage[strict]{csquotes}
\usepackage{alphabeta}              % in LuaLaTeX use {unicode-math}
\usepackage[fleqn,leqno]{mathtools} % loads amsmath, makes \colonequals from :=
\usepackage[fleqn,leqno]{amsmath}   % not needed with amsart, needed by article
\usepackage[acronym,toc,nomain]{glossaries}  % for a list of acronyms
\usepackage{hyperref}               % options in hyperref.cfg & personal.sty
\usepackage[capitalise]{cleveref}   % _after_ hyperref & amsmath. Use: in \Cref{t:1}
\usepackage[dvipsnames]{xcolor}     % used by \hypersetup
\usepackage[ocgcolorlinks]{ocgx2}[2017/03/30]	% linebreak links and uncolor in print

   \let\digamma\relax                   % in front of amsfonts
\usepackage{amsfonts}
\usepackage{amsthm}
\usepackage{booktabs}
\usepackage{changepage}
\usepackage{diagbox}
\usepackage{dynkin-diagrams}
\usepackage{enumitem}
\usepackage{parskip}
\usepackage{pgfplots}
\pgfplotsset{compat=1.16}     % less compatibility
\usepackage{rotating}
\usepackage{siunitx}
\usepackage{subcaption}
\usepackage{thmtools}
\usepackage{tikz}
\usepackage{tikz-cd}
\usetikzlibrary{babel}
\usepackage[colorinlistoftodos, textsize = footnotesize]{todonotes}
\usepackage{xparse}
\usepackage{xspace}
\usetikzlibrary{calc,tqft,decorations.markings}

\usepackage[backend=biber,maxbibnames=99,doi=false]{biblatex}

%% file: configuration.tex
\newtoggle{nome}
\toggletrue{nome}
\mathtoolsset{centercolon}

\makeatletter
\def\paragraph{\vspace{6pt}\@startsection{paragraph}{4}%
  \z@\z@{-\fontdimen2\font}%
  {\normalfont\bfseries}}
\makeatother

\newcounter{todocounter}
\DeclareDocumentCommand\addreference{g}{\stepcounter{todocounter}\todo[color = blue!30]{\thetodocounter. Add reference\IfNoValueF{#1}{: #1}}\xspace}
\DeclareDocumentCommand\checkthis{g}{\stepcounter{todocounter}\todo[color = red!50]{\thetodocounter. Check this\IfNoValueF{#1}{: #1}}\xspace}
\DeclareDocumentCommand\fixthis{g}{\stepcounter{todocounter}\todo[color = orange!50]{\thetodocounter. Fix this\IfNoValueF{#1}{: #1}}\xspace}
\DeclareDocumentCommand\expand{g}{\stepcounter{todocounter}\todo[color = green!50]{\thetodocounter. Expand\IfNoValueF{#1}{: #1}}\xspace}

\tikzset{every loop/.style={min distance=1cm, looseness=30}}
\tikzset{vertex/.style={circle, draw, black, minimum size=6pt, inner sep=0pt}}
\tikzset{half-vertex/.style={semicircle, draw, fill, black, minimum size=3pt, inner sep=0pt, yshift=1.5pt}}

\tikzset{middlearrow/.style={decoration={markings, mark= at position 0.5 with {\arrow{#1}},}, postaction={decorate}}}

\declaretheoremstyle[
  spaceabove = 3pt,
  spacebelow = 3pt,
  bodyfont = \itshape,
]{first}
\declaretheorem[numberwithin=section, style=first]{theorem}
\declaretheorem[sibling=theorem, style=first]{conjecture}
\declaretheorem[sibling=theorem, style=first]{corollary}
\declaretheorem[sibling=theorem, style=first]{lemma}
\declaretheorem[sibling=theorem, style=first]{proposition}

\declaretheorem[numberwithin=section, style=first, title=Theorem]{alphatheorem}

\declaretheorem[sibling=alphatheorem, style=first, title=Corollary]{alphacorollary}
\declaretheorem[sibling=alphatheorem, style=first, title=Proposition]{alphaproposition}

\declaretheoremstyle[
  spaceabove = 0pt,
  spacebelow = 0pt,
]{second}
\theoremstyle{second}

\declaretheorem[style=second, sibling=theorem]{definition}
\declaretheorem[style=second, sibling=theorem]{example}

\declaretheorem[style=second, sibling=theorem]{remark}

\crefname{conjecture}{Conjecture}{Conjectures}
\crefname{construction}{Construction}{Constructions}

\crefname{alphatheorem}{Theorem}{Theorems}
\crefname{alphaconjecture}{Conjecture}{Conjectures}
\crefname{alphacorollary}{Corollary}{Corollaries}
\crefname{alphaproposition}{Proposition}{Propositions}

\crefname{figure}{Figure}{Figures}

\makeatletter
\def\gitfootnote{\gdef\@thefnmark{}\@footnotetext}
\makeatother

\makeatletter
\DeclareRobustCommand\widecheck[1]{{\mathpalette\@widecheck{#1}}}
\def\@widecheck#1#2{%
    \setbox\z@\hbox{\m@th$#1#2$}%
    \setbox\tw@\hbox{\m@th$#1%
       \widehat{%
          \vrule\@width\z@\@height\ht\z@
          \vrule\@height\z@\@width\wd\z@}$}%
    \dp\tw@-\ht\z@
    \@tempdima\ht\z@ \advance\@tempdima2\ht\tw@ \divide\@tempdima\thr@@
    \setbox\tw@\hbox{%
       \raise\@tempdima\hbox{\scalebox{1}[-1]{\lower\@tempdima\box
\tw@}}}%
    {\ooalign{\box\tw@ \cr \box\z@}}}
\makeatother

\makeindex
\crefformat{enumi}{#2\textup{(#1)}#3}

\makeatletter
\newcommand{\acro}[3]{%
  \newacronym{#1}{#2}{#3}%
  \@namedef{#1}{{\acrshort{#1}}\xspace}%
  \@namedef{#2}{{\acrlong{#1}}\xspace}%
}

\acro{qft}{QFT}{Quantum Field Theory}
\acro{tqft}{TQFT}{Topological Quantum Field Theory}
\acro{cft}{CFT}{Conformal Field Theory}
\acro{rcft}{RCFT}{Rational Conformal Field Theory}
\acro{sft}{SFT}{Symplectic Field Theory}
\acro{cohft}{CohFT}{Cohomological Field Theory}
\acro{trft}{TropFT}{Tropical Field Theory}

\acro{syz}{SYZ}{Strominger--Yau--Zaslow}
\acro{tuy}{TUY}{Tsuchiya--Ueno--Yamada}
\acro{wdvv}{WDVV}{Witten--Dijkgraaf--Verlinde--Verlinde}
\acro{wz}{WZ} {Wess--Zumino--Witten}

\acro{glsm}{GLSM}{Gauged Linear Sigma Model}

%% file: macros.tex
\newcommand   \graph {\gamma}          % a graph,  Kohno-style,  != \iter
\newcommand \sphere  {\ensuremath{S}}         % sphere not always = surface
\DeclareMathOperator\SL{SL}                   % special linear group,     SL(2)
\DeclareMathOperator\SO{SO}                   % special orthogonal group, SO(3,R)
\DeclareMathOperator \Gr     {Gr}               % Grassmannian (variety)
\DeclareMathOperator \OGr    {OGr}              % isotropic/orthogonal Gr
\DeclareMathOperator \LGr    {LGr}              % isotropic/Lagrangian Gr = SGr(*,2*)
\DeclareMathOperator \Pic    {Pic}              % Picard (group scheme)
\DeclareMathOperator \moduli {M}                % moduli (space)
\newcommand \even {\ensuremath{\moduli_C(2,\cO_C)}} % det = \cO, trivial
\newcommand \odd  {\ensuremath{\moduli_C(2,\cL)}}   % det = L, fixed of odd degree

\DeclareMathOperator \Fuk {Fuk}                 % -//- majuscule v.
\mathchardef\mhyphen="2D
\newcommand\dash{\nobreakdash-\hspace{0pt}}

\newcommand \manontoric [1] {\ensuremath{Y_{P_{#1},\cM_{#1}}}}   % Manon toric variety
\newcommand \gptoric    [3] {\ensuremath{X_{P_{#1,#2},#3_{#1}}}} % GP toric variety

\newcommand\pt{\ensuremath{\mathrm{pt}}}

\newcommand \dd    {\ensuremath{\mathrm{d}}}      % de Rham differential

\newcommand \interior  {\ensuremath{\mathrm{int}}}    % interior of a polytope
\newcommand \regular   {\ensuremath{\mathrm{reg}}}    % regular locus
\newcommand \singular   {\ensuremath{\mathrm{singular}}}    % regular locus
\newcommand \onecosk   {\ensuremath{\mathrm{1-cosk}}}    % 1-coskeleton

\DeclareMathOperator\CC{C}

\DeclareMathOperator\ev{ev}              % evaluation map
\DeclareMathOperator\fano{F}
\DeclareMathOperator\Fl{Fl}
\DeclareMathOperator\HH{H}

\DeclareMathOperator\RR{R}

\DeclareMathOperator\Spec{Spec}          % affine spectrum of a ring
\newcommand \bA {{\ensuremath{\mathbb{A}}}}

\newcommand \bC {{\ensuremath{\mathbb{C}}}}

\newcommand \bF {{\ensuremath{\mathbb{F}}}}

\newcommand \bP {{\ensuremath{\mathbb{P}}}}
\newcommand \bQ {{\ensuremath{\mathbb{Q}}}}
\newcommand \bR {{\ensuremath{\mathbb{R}}}}

\newcommand \bZ {{\ensuremath{\mathbb{Z}}}}

\newcommand \cE {\ensuremath{\mathcal{E}}}

\newcommand \cL {\ensuremath{\mathcal{L}}}
\newcommand \cM {\ensuremath{\mathcal{M}}}
\newcommand \cN {\ensuremath{\mathcal{N}}}
\newcommand \cO {\ensuremath{\mathcal{O}}}
\newcommand \cP {\ensuremath{\mathcal{P}}}

\newcommand \cV {\ensuremath{\mathcal{V}}}

\newcommand \rP {\ensuremath{\mathrm{P}}}

%% file: sympl-intro.tex
\section{Introduction}
\label{section:introduction}

We discuss various aspects of the mirror symmetry
for the moduli space~$\odd$ of rank~2 vector
bundles with fixed determinant $\mathcal{L}$ of odd degree
on an algebraic curve~$C$ of genus~$g\geq 2$.
These are smooth projective Fano varieties of dimension~$3g-3$
in the algebro-geometric context,
and monotone symplectic manifolds of dimension~$6g-6$
in the symplecto-geometric context
equipped with the Atiyah-Goldman-Narasimhan symplectic form,
where they are denoted~$\cN_g$.
They have been the subject of many works,
spanning 50+ years of research.

In \cite{gp-tqft} we have introduced \emph{graph potentials},
a class of Laurent polynomials with interesting symmetries,
associated to (colored) trivalent graphs.
The goal of this article is to relate these graph potentials to~$\odd$.
We do this by considering the monotone Lagrangian fiber~$L_\graph$
of Nishinou--Nohara--Ueda's integrable system \cite{MR2609019,MR2879356}
associated with Manon's toric degeneration for~$\odd$,
and identify the graph potential~$\smash{\widetilde{W}_{\graph,c}}$
with the Floer potential of~$L_\graph$.
This allows us to identify the classical period of the graph potential
with the quantum period of~$\odd$.
This is a manifestation of mirror symmetry for Fano varieties,
and for an excellent survey on this subject
one is referred to \cite{2112.15339}.

\paragraph{Graph potentials}
Let~$\gamma=(V,E,c)$ be a (colored) trivalent graph.
For every trivalent vertex~$v\in V$ with edges~$a,b,c\in E$
we define the \emph{vertex potential}
as the Laurent polynomial~$\smash{\frac{a}{bc}+\frac{b}{ac}+\frac{c}{ab}+\frac{1}{abc}}$
in the uncolored case,
and~$\smash{\frac{ab}{c}+\frac{ac}{b}+\frac{bc}{a}+abc}$ in the colored case,
with variables indexed by the edges.
For the \emph{graph potential} we take the sum over all vertices.
For a trivalent graph of genus~$g$
we obtain a Laurent polynomial in~$3g-3$ variables.
For more details on the construction one is referred to \cref{subsection:graph-potentials}.

Graph potentials have interesting symmetry properties,
where the choice of different trivalent graphs of fixed genus~$g$
gives rise to explicit mutations between the different potentials.
It is moreover possible to combine series
expansions of oscillating integrals of graph potentials
for all~$g$ simultaneously
into a topological quantum field theory.
These are the main results of \cite{gp-tqft}.
By establishing mirror symmetry for graph potentials and~$\odd$
it becomes possible to use the efficient computational methods
for classical periods of graph potentials in op.~cit.
and apply them to quantum periods of~$\odd$.

\paragraph{A monotone Lagrangian torus for $\odd$}
Our first goal is
to obtain an integrable system~$\Phi_{\graph,c}\colon \odd \to \Delta_{\graph,c}$
for every (colored) graph~$\graph$.
We do this in \cref{subsection:integrable-system}.
This crucially uses a degeneration of~$\odd$ into a
 toric variety constructed by Manon,
as recalled in \cref{subsection:manon}.
Then we use symplectic parallel transport for the natural
integrable system associated to a toric variety,
and obtain the following

\begin{alphaproposition}[\cref{proposition:monotone-torus}]\label{proposition:alpha1}
  For every trivalent graph~$\gamma$
  the Lagrangian torus
  \begin{equation}
    L_\graph := \Phi_{\graph,c}^{-1}(0) \subset \odd
  \end{equation}
  is monotone.
\end{alphaproposition}
By \cref{corollary:non-Hamiltonian-isotopic} these tori
are non-Hamiltonian isotopic
to each other, for different choices of~$\graph$.
%In \cref{remark:speculation} and the surrounding discussion
%we speculate on how the different Lagrangian tori $L_\graph$
%for different choices of trivalent graphs $\graph$ are related.

The next ingredient is motivated by a foundational
result of Nishinou--Nohara--Ueda,
\cite[Theorem~1]{MR2879356} and recalled in \cref{theorem:nnu}.
It gives a recipe to compute the Floer potential of a
smooth projective Fano variety
using a toric degeneration.
Our version gets rid of some of the
technical conditions in the original result which could be hard to check in practice.
\begin{alphatheorem}[\cref{theorem:gennnu}]
  \label{theorem:magic-lemma}
  Let~$X_1$ be a smooth projective Fano variety
  which admits a degeneration into a Gorenstein
  Fano toric variety~$X_0$. Assume that the degeneration
  preserves the second Betti number.
  Let~$L_t(u_0)$ be a monotone Lagrangian torus
  obtained by symplectic parallel transport. Then
  the Newton polytope of $m_0(L)$ equal the
  polar dual $P^{\circ}$ of the moment polytope
  of the toric variety $X_0$.

  In particular, if the polytope $P^{\circ}$ has non-zero
   lattice points other than the vertices
   i.e. $X_0$ has terminal singularities,
   then the Floer potential can be expressed as
  \begin{equation}
    m_0(L)(x) = \sum_{i=1}^m \exp(\langle v_i,x \rangle) T^{\ell_i(u_0)},
  \end{equation}
  summing over the the generators of the polar dual of the moment polytope,
  where~$\ell_i$ are the defining equations for the facets.
\end{alphatheorem}
%The more precise version of the statement is given in \cref{theorem:gennnu}.
%This is the key ingredient in the proof of mirror
% symmetry for $\odd$ and graph potentials,
%and we expect it can be used to study other instances.
This is the key ingredient in establishing the relationship
between $\odd$ and graph potentials as \emph{mirror partners}
and we expect it can be used to study other instances.

Applying this to our setup
we identify the Floer potential of the torus $L_\graph$
with the graph potential,
as in \cref{theorem:graphisdisc}.
In \cref{remark:second-betti} we will discuss a variation of the statement.
%This is related to the fact that the enumerative invariants counting holomorphic spheres (disks)
%passing through a point (with boundary on $L_\graph$)
%can be computed tropically: the respective curves can be degenerated
%to the regular locus of the singular toric variety
%and lifted to the complement of the exceptional locus on the resolution.

\paragraph{Classical periods and quantum periods}
This identification allows us to interpret the constant
 term of the $d$th power of the graph potential
as the number of maps from tropical trees to $\Delta_{\graph,c}$
in tropical descendant Gromov--Witten
invariant enumeration \cite{Rau-phd},
which in turn can be shown to agree with
the descendant Gromov--Witten invariant of $\odd$.
In particular, this proves the relation
between closed and open Gromov--Witten invariants
found and proved in the context of toric degenerations
independently by Bondal--Galkin \cite{Bondal-Galkin}
 and Nishinou--Nohara--Ueda \cite{Nishinou-Nohara-Ueda}.
Recently Tonkonog
proved a generalization of this relation for the case
of an abstract monotone Lagrangian torus \cite{1801.06921}.

Such a correspondence between
 the classical period of a Laurent polynomial
and the Gromov--Witten invariants of a Fano variety
goes by the name of \emph{mirror partner} \cite[\S3]{2112.15339}.
The main result we prove is the following
\begin{alphatheorem}
  \label{theorem:mirror-symmetry-introduction}
  Let~$(\gamma,c)$ be a trivalent colored graph of genus~$g$
  such that the total number of colored vertices is odd.
  Let~$C$ be a smooth projective curve of genus~$g$.
  Then the graph potential
  is a mirror partner of~$\odd$.
\end{alphatheorem}
From this identification between the classical period and
the regularized quantum period
we can also deduce the regularity of
regularized quantum periods,
see \cref{corollary:ode}.
\paragraph{Optimal monotone tori}
Let $X$ be any Fano manifold and let $\ast_0$ denote
the quantum multiplication in the small quantum
cohomology ring of $X$.
From Galkin--Golyshev--Iritani \cite{MR3536989}
 consider
\begin{equation}
  T_X:=\text{max}\{ |u|: u \in \mathbb{C} \ \text{is an eigenvalue of} \ c_1(X)\ast_0\} \in \overline{\mathbb{Q}}
\end{equation}
The invariant $T_X$ is a symplectic invariant of the manifold $X$.

Let~$L$ be a monotone Lagrangian torus in~$X$.
Consider the number $T_X':=\limsup_n ([m_0(L)^n]_{x^0})^{\frac{1}{n}}$,
where $[(m_0(L)^n]_{x^0}$ denotes the constant term of
the~$n$th power of the Laurent polynomial $m_0(L)(x)$.
By definition, the weighted number~$m_0(L)(1,\dots,1)$ of
holomorphic Maslov index 2 discs with boundary on a
monotone Lagrangian torus $L$ is bounded below by $T_X'$.

Observe that~$m_0(L)(1,\dots,1)$ also is equal to the value
$T_{\mathrm{con}}(L)$ of the disk potential
at the unique Morse point (also known as
\emph{conifold point}) of the Floer potential $m_0(L)$
in the domain $\mathbb{R}_{+}^{\dim L}\subset (\mathbb{C}^\times)^{\dim L}$.
 Since~$T'_X$ equals \cite{MR3536989,MR4384381} % GGI , GI
 to the inverse of the radius of convergence of the
 {\em regularized quantum period}
  function of~$X$, we get $T'_X=T_X$. Hence we have an inequality
\begin{equation}
  \label{equation:conifoldinequality}
 m_0(L)(x)= T_{\mathrm{con}}(L)\geq T_{X}
\end{equation}
and we say $L$ is \emph{optimal} if this inequality is an equality.

Let $\mathcal{N}_g$ be the symplectic manifold
whose algebraic model is $\odd$.
This Fano symplectic manifold is also
popularly known as
the rank two character variety and it is
equipped with the Atiyah--Bott--Goldman--Narasimhan symplectic form.
Muñoz \cite{MR1695800} has proved
that $T_{\mathcal{N}_g}=8g-8$.
As a corollary to \cref{theorem:mirror-symmetry-introduction}
combined with a Torelli-type theorem, we show that
\begin{alphacorollary}
  \label{corollary:optimal} (\cref{theorem:graphisdisc}
  and \cref{corollary:non-Hamiltonian-isotopic})
  For any trivalent colored $\gamma$ graph of genus
  $g$ with one colored vertex,
  the monotone Lagrangian torus $L_{\gamma}$ obtained
  as in \cref{proposition:alpha1} is optimal.
  Moreover for each pair of trivalent graphs $(\gamma, \gamma')$,
  the tori $L_{\gamma}$ and $L_{\gamma'}$ are not Hamiltonian isotopic.
\end{alphacorollary}

\begin{remark}
\Cref{corollary:optimal} highlights a key feature of
  the graph potentials and hence of the toric degenerations
   and the monotone Lagrangian tori we obtained.
    We are not aware of any examples of Fano manifolds
    with multiple optimal monotone Lagrangian tori.
\end{remark}
\paragraph{On mirror symmetry for Fano varieties}
The mirror dual of~$\odd$ is expected to be
a cluster-like variety equipped with a regular function,
the so-called Landau--Ginzburg potential
(closely related to the Floer potential),
that will play the central rôle in this story.
The graph potentials form a class of
possible hierarchies of such functions
adapted to the ``collective'' study of mirror symmetry
for the moduli spaces of local systems
on surfaces of all genera and all numbers of boundary components
simultaneously, yet still keeping the rank fixed.

There are a couple of interdependent theoretical
approaches to mirror symmetry for Fano varieties.
These provide different methods,
levels of rigour, constructivity, precision, predictive power,
scope of interests and applications.
But in what concerns experimental data,
so far in any approach mirror symmetry was mostly studied
either in low dimensions,
or for (Schubert cycles in) homogeneous varieties,
or for toric varieties,
or for their products (and some fibrations),
or for complete intersections therein,
see e.g.~\cite{MR2257391,MR3783413,MR3881789,MR3135701,MR3400560}.

On the other hand, moduli spaces
of rank~2~bundles
are Fano varieties which fall outside
 the usual scope of these methods,
and which are notorious for their falsifications
of naive conjectural extrapolations
 based on the data obtained from the examples above.
For example, in order to apply techniques
 such as the abelian/non-abelian correspondence
of \cite{MR2110629,MR2367022,MR2369087,MR2388562}
to compute quantum periods (and Givental $J$-functions),
the classification of
%Fano, Apéry, Iskovskikh and Mori--Mukai in terms of \git data,
Fano 3-folds was reworked in terms of GIT data in \cite[Theorem~A]{MR3470714}.
This theorem, which is proved in the first 105~sections of the paper,
says that all fibers in a family of Fano threefolds (and del Pezzo surfaces)
can be obtained as zero loci of regular sections
of a fixed vector bundle on a fixed ``key variety'',
which is isomorphic to a product of a toric variety
with a Grassmannian.
As a by-product this theorem implies that
\emph{moduli spaces of complex structures
on such Fano manifolds are unirational}.

In contrast, the (last) Section~106 of op.~cit.\ considers $\odd$:
by the non-abelian Torelli theorem (see e.g.~\cite[Theorem~E]{MR1336336})
the moduli space of complex structures on $\odd$ has a dominant rational
map to the Deligne--Mumford moduli space of stable curves.
For $g\geq 23$ these have Kodaira dimension at least~0
by \cite{MR0664324}.
This indicates that the mirror map
to the moduli space of complex structures on the Fano manifold $\odd$
from the complexified Kähler cone
%\footnote{I.e. an open domain in second cohomology with complex coefficients,
%that consists of classes $B+\ii\omega$ % (alternatively, $\omega + \ii B$)
%with the class (of the symplectic form) $\omega$ being Kähler
%and the class (of the so-called $B$-field) $B$ being arbitrary real.
%Thanks to Darboux--Moser--Weinstein theory it has an interpretation
%of the moduli space of symplectic structures,
%however it is closer to ``a Teichmüller space''.}
establishing mirror symmetry
has to be a transcendental map,
very different from all the known examples of mirror maps for Fano manifolds.

In \cite{gp-decomp} we will discuss aspects of
homological mirror symmetry
for~$\odd$ and graph potentials,
in order to understand certain (conjectural) decompositions
of invariants of~$\odd$
and graph potentials,
in a motivic,
symplectic
and categorical setting.

\paragraph{Complements}
The toric degeneration we have studied
allows us to recover a result of Kiem--Li on the
 terminality of the singularities of~$\even$,
at least for a generic curve~$C$.
We also state a conjecture on existence of
a small resolution of the toric degenerations
for~$\odd$ and~$\even$.
We refer the reader to \cref{theorem:terminal-criterion}, \cref{cor:KiemLi} and \cref{conjecture:small-resolution} for precise statements.

In \cref{section:explicit-calculations}
we will show how to compute classical periods
of graph potentials via combinatorial means,
or use algebro-geometric methods to compute
some quantum periods.
This shows how certain computations of periods
 can be performed on either side of the mirror.

\paragraph{Acknowledgements}
We want to thank
Najmuddin Fakhruddin,
Grigory Mikhalkin,
Paul Seidel
and Maxim Smirnov
for interesting discussions.

The frameworks that we use to establish mirror symmetry was set up
in an unpublished joint work of Alexey Bondal and S.G. \cite{Bondal-Galkin}.
Our treatment of Maslov index computations and holomorphic discs
uses constructions from joint works of S.G. and Grigory Mikhalkin,
\cite{Galkin-Mikhalkin-hq,Galkin-Mikhalkin,Galkin-Mikhalkin-2}
slightly adapted for our more special situation.

This collaboration started in Bonn in January--March 2018 during
the second author's visit to the
``Periods in Number Theory, Algebraic Geometry and Physics'' Trimester Program
of the Hausdorff Center for Mathematics (HIM)
and the first and third author's stay in
the Max Planck Institute for Mathematics (MPIM),
and the remaining work was done in
the Tata Institute for Fundamental Research (TIFR)
during the second author's visit in December 2019--March 2020
and the first author's visit in February 2020.
We would like to thank
HIM, MPIM and TIFR
for the very pleasant working conditions.
S.G. thanks Faculty of Mathematics of HSE University for
providing a sabbatical leave in 2018
and Mathematics Department of PUC-Rio for summer leave
in 2019-2020, which made possible the two visits above.

The initial progress on different parts of this work
was reported at
``Patchworking in Geometry and Topology'' conference in Belalp
(dedicated to the 70th anniversary of Oleg Viro)
and ``Geometry and Topology motivated by Physics''
at the Monte Verità Conference center in Ascona (Ticino).
We thank the partitipants for providing us with related references and context.
%The feedback from the participants greatly widened our understanding.
% Add other talks - MFO, Bengaluru, Swarnava

Several computations and experiments
were performed using
Pari/GP \cite{parigp}
and
Sage \cite{sagemath}
(in particular we used PALP \cite{palp} routines to analyze lattice polytopes).

The first author was partially supported by the FWO (Research Foundation---Flanders).
The second author was partially supported by CNPq grant PQ 315747.
The third author was partially supported by
the Department of Atomic Energy, India, under project no. 12-R\&D-TFR-5.01-0500
and also by the Science and Engineering Research Board, India (SRG/2019/000513).

%% file: sympl-toric.tex
% !TEX encoding = UTF-8 Unicode
\section{Toric degenerations of moduli spaces of vector bundles}
%\section{Toric degenerations}
\label{section:toric}
Abstracting from his earlier work with
Ciocan--Fontanine, Kim and van Straten \cite{MR1756568,MR1619529}
on mirror symmetry for Grassmannians and partial flag varieties,
Batyrev proposed in \cite{MR2112571,MR3289338}
a passage from
a ``small'' toric degeneration $X$
of a Fano variety~$F$
to a Laurent polynomial~$W$,
as a construction of a mirror dual to $F$.
See also \cite{2112.15339} for more context.
%%%
%% Given a toric degeneration $X\mapsto X_\delta$,
% One constructs
% a \LG model~$f: Y\to\A^1$
% which is mirror to the Fano variety $X$,
% starting with a Laurent polynomial~$W_\delta$
% which will have to coincide with~$f$
% on a torus~$(\Gm)^{\dim X}\subseteq Y$.
% XXX: certainly Batyrev did not propose this.
% XXX: The commented part above/below is misleading here and misplaced.
% XXX: It shall be moved to 'charts' discussion.
% For various questions
% in enumerative mirror symmetry
% the relationship between a Fano variety
% and its mirror Laurent polynomial
% (for which there are many choices)
% is sufficient, and we will discuss
% how the graph potentials we have introduced
% are mirror to~$\odd$.
% For the more general picture of mirror symmetry
% for this Fano variety
% one is referred to \cref{subsection:sod-vs-hms}.
%%

The toric degeneration for $\odd$ we are interested in
is related to the graph potentials introduced in \cite{gp-tqft}.
After recalling their construction and notation
we will describe a class of toric degenerations of~$\odd$ (and~$\even$)
introduced by Manon \cite{MR2928457}.
These two objects turn out to be closely related:
the Newton polytope of the graph potential associated to a trivalent graph~$\graph$
agrees with the moment polytope of Manon's degeneration
corresponding to the~$\graph$.

\subsection{Graph potentials}
\label{subsection:graph-potentials}
Let~$\graph=(V,E)$ be an undirected trivalent graph of genus~$g$,
possibly with loops.
We have that~$\#V=2g-2$ and~$\#E=3g-3$.
We will moreover assume it is connected.

We set
\begin{equation}
  \left\{
  \begin{aligned}
    \widetilde{N}_\graph&:=\CC^1(\graph,\mathbb{Z}) \\
    \widetilde{M}_\graph&:=\CC_1(\graph,\mathbb{Z})
  \end{aligned}
  \right.
\end{equation}
the free abelian groups of~1\dash(co)chains.

For every vertex~$v\in V$ we have the~3~edges~$e_{v_k},e_{v_l},e_{v_m}\in E$ adjacent to it.
These span a sublattice~$\widetilde{N}_v$ of~$\widetilde{N}_\graph$,
generated by the cochains~$x_i,x_j,x_k$ where we take~$x_i(e_{v_a})=\delta_{i,v_a}$.
For every~$v$ we can then consider the sublattice~$N_v$ spanned by the cochains~$\{\pm x_i\pm x_j\pm x_k\}$,
and we set
\begin{equation}
  N_\graph:=\operatorname{im}\left( \bigoplus_{v\in V}N_v\to\widetilde{N}_\graph \right).
\end{equation}
and
\begin{equation}
  M_\graph:=N_\graph^\vee.
\end{equation}
Finally we set
\begin{equation}
  A_\graph:=(\widetilde{N}_\graph/N_\graph)^\vee\cong M_\graph/\widetilde{M}_\graph
\end{equation}
which by \cite[Lemma~2.1]{gp-tqft} is isomorphic to~$\HH_1(\graph,\mathbb{F}_2)\cong\mathbb{F}_2^{\bigoplus g}$.
These are the relevant lattices,
relative to which all constructions in toric geometry are to be considered.

Associated to the different lattices we have algebraic tori
\begin{equation}
  \left\{
  \begin{aligned}
    T_\graph^\vee&:=\Spec\bC[N_\graph] \\
    \widetilde{T}_\graph^\vee&:=\Spec\bC[\widetilde{N}_\graph]
  \end{aligned}
  \right.
\end{equation}
Hence~$N_\graph$ and~$\widetilde{N}_\graph$ are the character lattice of~$T_{\graph}^{\vee}$ and~$\smash{\widetilde{T}_{\graph}^{\vee}}$,
and the cocharacter lattices of~$T_\graph$ and~$\smash{\widetilde{T}_\graph}$.
The group~$A_\graph$ is the kernel of the isogeny~$\smash{\widetilde{T}_\graph^\vee\to T_\graph^\vee}$.

The next ingredient is a coloring,
which is a function~$c\colon V\to\mathbb{F}_2$.
In the graph (co)homology language used above
this corresponds to a 0-chain on~$\graph$ with values in~$\mathbb{F}_2$.
If~$c(v)=0$ we will say that~$v$ is \emph{uncolored},
and if~$c(v)=1$ we say that~$v$ is \emph{colored}.

\begin{definition}
  Let~$(\graph,c)$ be a colored trivalent graph,
  with an enumeration of the edges given by~$e_1,\ldots,e_{3g-3}$,
  and let~$x_i$ be the coordinate function in~$\mathbb{Z}[\widetilde{N}_\graph]$ corresponding to~$e_i$.

  \begin{itemize}
    \item The \emph{vertex potential} for a vertex~$v\in V$ adjacent to the edges~$e_i,e_j,e_k$
      is defined as the Laurent polynomial
      \begin{equation}
        \label{equation:vertex-potential}
        \widetilde{W}_{v,c}:=
        \hspace{-1em}\sum_{\substack{(s_i,s_j,s_k)\in\bF_2^{\oplus 3} \\ s_i+s_j+s_k=c(v)}}\hspace{-1em}x_i^{(-1)^{s_i}}x_j^{(-1)^{s_j}}x_k^{(-1)^{s_k}}
        \in\mathbb{Z}[\widetilde{N}_v].
      \end{equation}
    \item The \emph{graph potential} of~$(\graph,c)$
      is defined as the Laurent polynomial
      \begin{equation}
        \widetilde{W}_{\graph,c}
        :=
        \sum_{v\in V}\widetilde{W}_{v,c}\in\mathbb{Z}[\widetilde{N}_\graph].
      \end{equation}
  \end{itemize}
\end{definition}
By \cite[Lemma~2.4]{gp-tqft}
the regular function~$\widetilde{W}_{\graph,c}$
on the torus~$\widetilde{T}_\graph^\vee$
descends to a function~$W_{\graph,c}$
on the torus~$T_\graph^\vee$.
We will study toric varieties with cocharacter lattice~$N_{\graph}$
whose moment polytope is the polar dual of the Newton polytope of~$W_{\graph,c}$.
However the lattice~$N_{\graph}$ has no natural basis,
hence we will not express~$W_{\graph,c}$ in terms of a basis of $N_{\graph}$.
%Instead we will work with~$\widetilde{W}_{\graph,c}$
%and reinterpret everything we need for applications in terms of~$W_{\graph,c}$.
We refer to \cite[\S2]{gp-tqft} for more details and examples.

By \cite[Corollary~2.9]{gp-tqft} the graph potential
depends only on the homology class of~$[c]$ in~$\HH_0(\graph,\mathbb{F}_2)$,
up to biregular automorphism of the torus
(hence under the connectedness assumption only on the parity).

\paragraph{Elementary transformations of graph potentials}
The second invariance result corresponds to the relation between
graph potentials for different choices of trivalent graphs (of fixed genus).
As explained in \cite[\S2.2]{gp-tqft} these correspond to different pair of pants decompositions
of a surface of genus~$g$,
and these are related via Hatcher--Thurston moves.
We are only interested in the operation from \cite[Figure~2]{gp-tqft},
and call this an \emph{elementary transformation}.
This is a procedure that transforms a colored trivalent graph~$(\graph,c)$
into a new colored trivalent graph~$(\graph',c)$.

We will now briefly recall the formulas for an elementary transformation of graph potentials,
for more details the reader is referred to \cite[\S2.2]{gp-tqft}.
Because a transformation is applied to an edge~$e\in E$ corresponding to a variable~$x$
we can always split the graph potential as
\begin{equation}
  \widetilde{W}_{\graph,c}=\widetilde{W}_{\graph,c}^{\text{mut}}+\widetilde{W}_{\graph,c}^{\text{frozen}}
\end{equation}
with the mutated part involving the variables attached to the vertices~$v_1$ and~$v_2$
incident to the edge~$e$.
There are (with possible repetitions) five variables involved:
the variables~$a,b,x$ for the vertex~$v_1$
and the variables~$c,d,x$ for the vertex~$v_2$,
see also \cite[Figure~3]{gp-tqft}.
The frozen part of the graph potential is not changed, and can be ignored.

If the parity of the coloring of the vertices~$v_1$ and~$v_2$ is odd,
then we have
\begin{equation}
  \begin{aligned}
    \widetilde{W}_{\graph,c}^{\text{mut}}
    &=xcd + \frac{x}{cd} + \frac{d}{cx} + \frac{d}{cx} + \frac{1}{abx} + \frac{ab}{x} + \frac{ax}{b} + \frac{bx}{a} \\
    &=\frac{1}{x}\left( ab + \frac{1}{ab} + \frac{c}{d} + \frac{d}{c} \right) + x\left( cd + \frac{1}{cd} + \frac{a}{b} + \frac{b}{a} \right), \\
    \widetilde{W}_{\graph',c}^{\text{mut}}
    &=x'bd + \frac{x'}{bd} + \frac{b}{dx'} + \frac{d}{bx'} + \frac{1}{acx'} + \frac{ac}{x'} + \frac{c}{ax'} + \frac{x'}{ac} \\
    &=\frac{1}{x'}\left( ac + \frac{1}{ac} + \frac{b}{d} + \frac{d}{b} \right) + x'\left( bd + \frac{1}{bd} + \frac{a}{c} + \frac{c}{a} \right).
  \end{aligned}
\end{equation}

If the parity of the coloring of the vertices~$v_1$ and~$v_2$ is even,
then we have
\begin{equation}
  \begin{aligned}
    \widetilde{W}_{\graph}^{\text{mut}}
    &=xcd + \frac{x}{cd} + \frac{c}{dx} + \frac{d}{cx} + abx + \frac{a}{bx} + \frac{x}{ab} + \frac{b}{ax} \\
    &=\frac{1}{x}\left( \frac{a}{b} + \frac{b}{a} + \frac{c}{d} + \frac{d}{c} \right) + x\left( cd + \frac{1}{cd} + \frac{1}{ab} + ab \right), \\
    \widetilde{W}_{\graph}^{\text{mut}}
    &=x'bd + \frac{x'}{bd} + \frac{b}{dx'} + \frac{d}{bx'} + acx' + \frac{a}{cx'} + \frac{c}{ax'} + \frac{x'}{ac} \\
    &=\frac{1}{x'}\left( \frac{b}{d} + \frac{d}{b} + \frac{a}{c} + \frac{c}{a} \right) + x\left( bd + \frac{1}{bd} + ac + \frac{1}{ac} \right).
  \end{aligned}
\end{equation}

By \cite[Theorems~2.12 and 2.13]{gp-tqft}
the graph potentials~$\widetilde{W}_{\graph,c}$ and~$\widetilde{W}_{\graph',c}$
are identified after a rational change of coordinates,
which is moreover invariant under the actions of the finite groups~$A_\graph$ and~$A_{\graph'}$.
We will come back to elementary transformations in the discussion surrounding the nature of mutations of Lagrangian tori,
see \cref{remark:speculation} and discussions preceeding it. 

Recall that given any Laurent polynomial $W$, the $m$-th classical period $c_m$ is defined to be the coefficient of the constant term of $W^m$. The following theorem relates the period of the graph potential $(\gamma,c)$ and $(\gamma',c')$
\begin{theorem}(Graph potential TQFT)\label{theorem:graphtqft} Let $(\gamma,c)$ and $(\gamma',c')$ be colored trivalent graphs of genus $g$ with same parity of the coloring, then for any positive integer $m$, the classical periods of the associated graph potentials agree.
	\end{theorem}

\subsection{Manon's degeneration}
\label{subsection:manon}
In \cite{MR2928457} Manon studied
the homogeneous coordinate rings
of the moduli \emph{stack} $\cM_{C;p_1,\ldots,p_n}(\SL_2)$
of quasi-parabolic bundles of rank~2.
The Picard group of this stack is~$X(B)^n\times\bZ$
where~$X(B)=\bZ\omega_1$ is the character group
of the Borel subgroup of~$\SL_2$.
Manon studies the algebraic properties
of the various graded algebras
one obtains by choosing a line bundle on the stack.

For us it will suffice to take~$n\leq 1$,
and hence we are interested in
the moduli stacks~$\cM_C(\SL_2)$
and~$\cM_{C,p}(\SL_2)$.
Then we denote\index{R@$R_C(b)$}\index{R@$R_{C,p}(a,b)$}
\begin{equation}
\begin{aligned}
R_{C}(b)     &:= \bigoplus_{\ell\geq 0}\HH^0(\cM_{C}  (\SL_2),\cL(b)^{\otimes\ell}) \\
R_{C,p}(a,b) &:= \bigoplus_{\ell\geq 0}\HH^0(\cM_{C,p}(\SL_2),\cL(a\omega_1,b)^{\otimes\ell})
\end{aligned}
\end{equation}
the graded algebras associated to the corresponding line bundles.

Using the theory of conformal blocks
and the Rees algebra construction
Manon defines
a \emph{flat degeneration of these graded algebras}
to the boundary~$\overline{\cM}_{g,n}\setminus\cM_{g,n}$
of the Deligne--Mumford compactification of the stack of pointed curves.
This stack has a stratification by closed substacks,
indexed by the weighted dual graphs of the stable curves.
The most degenerate curves one gets (with only rational components)
correspond to trivalent graphs of genus~$g$ and~$n$ half-edges
with all weights being zero,
corresponding to pair of pants decompositions we discussed earlier.

The algebras obtained in this way
are semigroup algebras associated to polytopes,
hence we get toric varieties.
To understand the relationship
between these algebras
and the varieties~$\even$
and~$\odd$,
recall that by Drézet-Narasimhan~\cite[Théorème~B]{MR0999313}
the Picard groups of these varieties are generated by the class of the Theta divisor.
By \textcite[Theorem~8.5 and Theorem~9.4]{MR1289330} % BL
(and more generally the works \cite{MR1394754,MR1289830,MR3982698})
we have identifications
\begin{equation}
  \begin{aligned}
    R_{C}(2)
    &\cong\bigoplus_{\ell\geq 0}\HH^0(\even,\Theta^{\otimes 2\ell}) \\
    R_{C,p}(2,2)
    &\cong\bigoplus_{\ell\geq 0}\HH^0(\odd,\Theta^{\otimes\ell}).
  \end{aligned}
\end{equation}
% of the homogeneous coordinate rings of these varieties
% with the polarization by the Theta divisor.
Hence we get a toric degeneration of the moduli spaces
we are interested in,
for every choice of trivalent graph~$\graph$ with at most one half-edge.
% We will make these statements more precise
% after introducing the relevant polytopes and dealing with half-edges accordingly.

\paragraph{Graphs from stable curves}
Let~$(C;p_1,\ldots,p_n)$ be a stable nodal curve with~$n$ marked points.
The \emph{dual graph} associated to~$(C;p_1,\ldots,p_n)$
is the weighted graph~$\graph=(V,E,g)$\index{Gb@$\graph=(V,E,g)$}
with~$n$ half-edges defined as
\begin{itemize}
  \item each vertex~$v_i\in V$ corresponds to an irreducible component~$C_i$ of~$C$, with weight~$g(v_i)$ the genus of~$C_i$;
  \item the number of edges between~$v_i$ and~$v_j$ is~$\#C_i\cap C_j$;
  \item for each marked point~$p$ in the component~$C_i$ we add a half-edge to the vertex~$v_i$.
\end{itemize}
The arithmetic genus of the curve~$C$ is then~$\#E-\#V+1+\sum_{v\in V}g(v)$.

We will only consider nodal curves with marked points~$(C;p_1,\ldots,p_n)$,
for which the dual graph~$\graph$ (containing half-edges)
is trivalent (and all weights are zero).
These correspond to the zero-dimensional strata in the stratification of~$\overline{\cM}_{g,n}\setminus\cM_{g,n}$
in terms of the type of the dual graph.
Equivalently, all components~$C_i$ are required to be rational.

We can rephrase the definition of Manon's polytope from \cite[Definition~1.1]{MR2928457} as follows.

\begin{definition}
  \label{definition:manon}
  Let~$\graph$ be a trivalent weighted graph of genus~$g$ with at most one half-edge. The polytope~$P_\graph\subseteq\bR^{\#E}$\index{P@$P_\graph$} is the set of non-negative real weightings of the edges of~$\graph$, such that
  \begin{itemize}
    \item the weights of half-edges are precisely~2;
    \item for every vertex~$v\in V$, if~$w_1(v),w_2(v),w_3(v)$ are the weights of the edges incident to it, then
      \begin{equation}
        2\max\{w_1(v),w_2(v),w_3(v)\}\leq w_1(v)+w_2(v)+w_3(v)\leq 4.
      \end{equation}
  \end{itemize}
  We define the lattice~$\cM_\graph$\index{M@$\cM_\graph$} using the lattice of integer points in~$\bR^{\#E}$ and the condition that~$w_1(v)+w_2(v)+w_3(v)\in2\bZ$ for every~$v\in V$.
\end{definition}

The following theorem then describes the sought-after toric degeneration, which is a special case of \cite[Theorem~1.2]{MR2928457} and \cite[Theorem~1.3]{MR3861806}.
\begin{theorem}[Manon]
  \label{theorem:manon}
  Let~$C$ be a smooth projective curve of genus~$g\geq 2$. Let~$\graph$ be a trivalent graph of genus~$g$, with at most one half-edge. Then
  \begin{itemize}
    \item if there are no half-edges in~$\graph$, then the homogeneous coordinate ring~$R_C(2)$ for~$\even$ can be flatly degenerated to the semigroup algebra associated to the polytope~$P_\graph$ in the lattice~$\cM_\graph$;
    \item if there is one half-edge in~$\graph$, then the homogeneous coordinate ring~$R_C(2,2)$ for~$\odd$ can be flatly degenerated to the semigroup algebra associated to the polytope~$P_\graph$ in the lattice~$\cM_\graph$.
  \end{itemize}
\end{theorem}

\begin{definition}
  The toric variety from \cref{theorem:manon} will be denoted~$\manontoric{\graph}$\index{Y@$\manontoric{\graph}$}. Observe that~$P_{\graph}$ is the moment polytope and $\cM_{\graph}$ is the character lattice.
\end{definition}

We will rewrite the inequality in \cref{definition:manon} so that the origin becomes an internal point of the polytope and the polytope becomes full-dimensional in the  character lattice of the toric variety. Setting~$u_i(v):= w_i(v)-1$ for the weight of the~$i$th edge at the vertex~$v$ we obtain the following lemma.
\begin{lemma}
  \label{lemma:inequalities-1}
  Let~$\graph$ be a trivalent graph of genus~$g\geq 2$ with at most one half-edge. The polytope~$P_\graph$ is defined by the inequalities
  \begin{itemize}
    \item if~$v$ is a vertex without half-edges, then
      \begin{equation}
        \label{equation:uncolored-inequalities}
        \begin{aligned}
          -u_1(v)-u_2(v)-u_3(v)&\geq-1 \\
          -u_1(v)+u_2(v)+u_3(v)&\geq-1 \\
          u_1(v)-u_2(v)+u_3(v)&\geq-1 \\
          u_1(v)+u_2(v)-u_3(v)&\geq-1;
        \end{aligned}
      \end{equation}
    \item if~$v$ is a vertex with a (necessarily unique) half-edge, then~$u_1=-u_2$ for the weights of the other two edges.
  \end{itemize}
\end{lemma}

\paragraph{Removing half-edges}
To make the link to the setup of \emph{colored} trivalent graphs from \cref{subsection:graph-potentials}
we will explain how to reduce trivalent graphs with one half-edge to colored trivalent graphs.
A more general procedure can be constructed, but in this paper we are only interested in the case with one half-edge.

Let~$\graph=(V,E)$ be a trivalent graph of genus~$g\geq 2$ with precisely one half-edge.
Let~$v\in V$ be the vertex to which the half-edge~$e$ is attached.
Let~$e_1$ and~$e_2$ be the other edges in~$v$,
which are incident to the vertices~$v_1$ and~$v_2$.
Then we construct a \emph{colored} trivalent graph~$(\graph',c)$ as follows:
\begin{itemize}
  \item the vertices~$V'$ are~$V\setminus\{v\}$;
  \item the edges~$E'$ are~$(E\cup\{e_{1,2}\})\setminus\{e,e_1,e_2\}$,
    where~$e_{1,2}$ is an edge between~$v_1$ and~$v_2$
  \item we color precisely one vertex from~$\{v_1,v_2\}$.
\end{itemize}

Graphically we have the situation from \cref{figure:collapsing}.
\begin{figure}[h]
  \centering
  \begin{tikzpicture}[scale=2]
    % before
    \node[vertex] (A) at (0,0)   {};
    \node[vertex] (B) at (1,0)   {};
    \node[vertex] (C) at (0.5,0) {};

    \draw (A) + (0,0.05) node [above] {$v_1$};
    \draw (B) + (0,0.05) node [above] {$v_2$};
    \draw (C) + (0,0.05) node [above] {$v$};

    \draw (A) edge node [above] {$e_1$} (C);
    \draw (B) edge node [above] {$e_2$} (C);

    \draw[decorate, decoration=snake] (C) -- node [right] {$e$} (0.5,-0.5);

    % arrow
    \node at (2,0) {$\mapsto$};

    % after
    \node[vertex, fill] (Anew) at (3,0) {};
    \node[vertex]       (Bnew) at (4,0) {};

    \draw (Anew) edge node [above] {$e_{1,2}$} (Bnew);
    \draw (Anew) + (0,0.05) node [above] {$v_1$};
    \draw (Bnew) + (0,0.05) node [above] {$v_2$};
  \end{tikzpicture}
  \caption{Removing a half-edge}
  \label{figure:collapsing}
\end{figure}
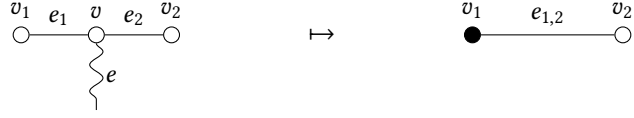

With this procedure we define a new polytope as follows.
\begin{definition}\label{definition:manon-polytope}
  Let~$\graph$ be a trivalent graph of genus~$g\geq 2$ with one half-edge.
  Let~$(\graph',c)$ be the colored trivalent graph
  obtained by removing the half-edge,
  where~$v_1\in V'$ is the colored vertex.
  The polytope~$P_{\graph',c}\subseteq\bR^{\#E'}$
  is defined by the same inequalities for all~$v\in V'\setminus\{v_1\}$,
  and in~$v_1$ we consider the inequalities
  \begin{equation}
    \label{equation:colored-inequalities}
    \begin{aligned}
      u_1(v_1)+u_2(v_1)+u_3(v_1)&\geq-1 \\
      u_1(v_1)-u_2(v_1)-u_3(v_1)&\geq-1 \\
      -u_1(v_1)+u_2(v_1)-u_3(v_1)&\geq-1 \\
      -u_1(v_1)-u_2(v_1)+u_3(v_1)&\geq-1.
    \end{aligned}
  \end{equation}
  Hence this polytope is defined in the subspace of~$\bR^{\#E}$
  given by the equality~$u_1(v_1)=u_2(v_2)$
  where~$u_1$ and~$u_2$ refer to the edges~$e_1$ and~$e_2$
  (recall that half-edges have a constant value assigned to them).
\end{definition}

\paragraph{The polytope~$P_{\graph,c}$}
Let~$(\graph,c)$ be a colored trivalent graph of genus $g$ with no half-edges.
We use~$\widetilde{M}_\graph$ and~$\widetilde{N}_\graph$ as in \cref{subsection:graph-potentials}.

Consider the half-spaces~$H_i^{c(v)}$ (for~$i=1,2,3,4$) defined by the inequalities
\begin{equation}
  \label{equation:half-spaces-general-i}
  \begin{aligned}
  	(-1)^{c(v)}(-u_1-u_2-u_3)&\geq -1,\\
  	(-1)^{c(v)}(-u_1+u_2+u_3)& \geq -1, \\
  	(-1)^{c(v)}(u_1-u_2+u_3)&\geq -1,\\
  	(-1)^{c(v)}(u_1+u_2-u_3)& \geq -1,
  \end{aligned}
\end{equation}
where~$u_1,u_2,u_3$ correspond to the three edges incident at the vertex~$v$.
Observe that the inequalities~\eqref{equation:half-spaces-general-i}
specialize to the inequalities~\eqref{equation:uncolored-inequalities}
(resp.~\eqref{equation:colored-inequalities}) if~$c(v)=0$ (resp.~~$c(v)=1$).

\begin{definition}
  Let $(\graph,c)$ be a colored trivalent graph with no half-edges. Let $V$ be the set of vertices and $E$ be the set of edges of $\graph$. Consider the polytope~$P_{\graph,c}$ defined by the intersection of half-edges in $\widetilde{M}_{\graph}$\index{P@$P_{\graph,c}$}
  \begin{equation}
    P_{\graph,c}=\bigcap_{v\in V}\left( H_1^{c(v)}\cap H_2^{c(v)}\cap H_3^{c(v)}\cap H_4^{c(v)}\right). %\right)\cap\left( \bigcap_{v\neq v_0}H_1^v\cap H_2^v\cap H_3^v\cap H_4^v \right).
  \end{equation}
\end{definition}

We can also describe the polar dual~$P_{\graph,c}^\circ \in \widetilde{N}_{\graph}$\index{P@$P_{\graph,c}^\circ$} as follows.
For~$v\in V$ a vertex with coloring~$c(v)$,
and~$e_i,e_j,e_k$ the edges adjacent to~$v$
we can consider the subgraph~$\graph_v$,
and define the vector space
\begin{equation}
  \label{equation:colored-cycles}
  C_{v,c}:=\{s\in\CC_1(\graph_v,\bF_2)\mid\dd(s)=c(v)\}.
\end{equation}
For~$v\in V$ and~$s\in C_{v,c}$ we can then define the point
\begin{equation}\label{equation:vertices}
  p(v,s):=(0,\ldots,0,(-1)^{s_i},0,\ldots,0,(-1)^{s_j},0,\ldots,0,(-1)^{s_k},0,\ldots,0)\in N_{\graph} \subset \bR^{\#E}
\end{equation}
with~$\pm1$ in positions~$i,j,k$.

The following standard lemma then makes
the description as a convex hull of the polytope we are interested in explicit.
\begin{lemma}
  \label{lemma:polytope-description}
  Let~$(\graph,c)$ be a colored trivalent graph of genus~$g\geq 2$ with no half-edges. %Let~$(\graph',c)$ be the associated colored trivalent graph.
  The polar dual~$P^{\circ}_{\graph,c}$ of the polytope~$P_{\graph,c}$
  is the convex hull of the points~$p(v,s)$, where~$v\in V$ and~$s\in C_{v,c}$.
\end{lemma}

This allows us to make the next definition.
\begin{definition}
  Let~$\graph,c$ be a colored trivalent graph of genus~$g$ with no half-edges.
  The \emph{graph potential toric varieties}~$X_{P_{\graph,c},M_{\graph}}$\index{X@$X_{P_{\graph,c},M_{\graph}}$}
  (respectively~$X_{P_{\graph,c}, \widetilde{M}_{\graph}}$)\index{X@$X_{P_{\graph,c}, \widetilde{M}_{\graph}}$}
  are defined as the toric varieties with moment polytope~$P_{\graph,c}$ in the character lattice~$M_{\graph}$ (respectively~$\widetilde{M}_{\graph}$).
\end{definition}
We will often refer to the polar dual $P^{\circ}_{\gamma,c}$ as the fan polytope of the toric variety~$X_{P_{\graph,c},M_{\graph}}$.
We can now compare Manon's toric degeneration to the toric varieties obtained from the graph potentials,
and show that they are the same.

\begin{proposition}
  \label{proposition:isomorphism-of-toric-varieties}
  Let~$\graph$ be a trivalent graph of genus~$g$,
  with at most one half-edge.
  Let~$(\graph',c)$ be the associated colored trivalent graph
  obtained from removing half-edges.
  There exists an isomorphism
  \begin{equation}
    \manontoric{\graph}\cong\gptoric{\graph'}{c}{M}.
  \end{equation}
\end{proposition}

\begin{proof}
  By \cite[Proposition~2.3.9]{MR2810322},
  for any~$\omega$ in the character lattice of a toric variety,
  the moment polytope $P$ and the polytope $\omega+P$ have the same normal fan,
  and in turn give isomorphic toric varieties.

  Faust--Manon \cite[\S3]{MR4025429}
  translates the moment polytope of $\manontoric{\graph}$ by $\omega=(2,\dots, 2) \in \cM_{\graph}$
  to get the origin as the unique interior point.
  Observe that we have the identification of the lattices~$\cM_{\graph}=2{M}_{\graph'}$.
  Hence, we have to scale down by a factor of~$2$.
  Starting from Manon's equations for the polytope $P_{\graph}$,
  we translate it by the vector $(1,1,\dots,1)$ in $M_{\graph'}$
  to get the equations of the reflexive polytope $P_{\graph',c}$.
  The result follows.
\end{proof}

In particular have we found another realization of Manon's toric degeneration
using an explicit full-dimensional reflexive polytope.

As an application of the description of the toric degeneration
in terms of the combinatorics of trivalent graphs
we will in \cref{section:singularities}
discuss when these toric degenerations have terminal singularities
or admit a small resolution.
These properties will depend on the choice of trivalent graph
used for the construction of the degeneration. 
Terminality will be crucially used in~\cref{section:potentials}.
%We relegate these results to an appendix:
%in \cref{section:potentials} we will prove comparison results
%that do not require these two conditions to hold.
%but in the light of earlier comparison results
%who require these properties to go through
%and for purely algebro-geometric reasons
%it is interesting to know when they hold.

%The final observation to make is that the graph potential,
%defined in terms of the lattice~$\widetilde{M}_{\graph}$
%can also be interpreted in terms of~$M_{\graph}$.
%\begin{remark}
%  Moreover, by construction the Newton polytope of graph potential $\widetilde{W}_{\graph,c}$ (also $W_{\graph,c}$) is the polytope $P^{\circ}_{\graph,c}$.
%\end{remark}

%% file: sympl-monotone.tex
\section{Monotone Lagrangian tori from graph potentials}
We will now consider~$\odd$ and the toric degenerations considered in \cref{subsection:manon}
as symplectic manifolds,
and use symplectic transport to construct monotone Lagrangian tori in~$\odd$
in \cref{proposition:monotone-torus}.
Every choice of trivalent graph leads to a toric degeneration,
and as explained in \cref{corollary:non-Hamiltonian-isotopic}
we obtain as many monotone Lagrangian tori which are not Hamiltonian-isotopic to each other
as we have trivalent graphs of genus~$g$.
The case~$\even$ is similar.

\subsection{Reminder on integrable systems}
\label{subsection:integrable-system}
Let us recall the notion of
an integrable system
on a symplectic manifold~$(M,\omega)$
of real dimension~$2n$.
\begin{definition}
\label{definition:integrable-system}
An \emph{integrable system} on~$(M,\omega)$ is
a collection of~$n$ functionally independent
real-valued~$C^\infty$ functions~$\{H_1,\ldots,H_n\}$
on the manifold~$M$ which are pairwise Poisson-commutative,
i.e.~$\{H_i,H_j\}=0$ for all~$i,j=1,\ldots,n$,
where~$\{-,-\}$ is the Poisson bracket induced by the symplectic form~$\omega$.
We will denote it~$\Phi\colon M\to\bR^n$.\index{P@$\Phi: (M,\omega)\to\bR^n$}
\end{definition}

By definition, an integrable system induces a Hamiltonian~$\bR^n$\dash action on~$M$. By the Arnold--Liouville theorem we have that any regular, compact, connected orbit of this action gives a Lagrangian torus in~$M$.

We will work with integrable systems on singular toric varieties. In that set-up we want the real-valued~$C^{\infty}$ functions to Poisson-commute on the smooth locus.

Natural examples of integrable systems are found in the theory of toric varieties.
Here we let~$X_0$ be a toric variety of complex dimension~$n$,
and consider the torus action~$(\bC^\times)^n$ on~$X_0$.
The moment map for the torus action
with respect to any torus-invariant K\"ahler form gives an integrable system.

\paragraph{Integrable systems from toric degenerations}
Nishinou--Nohara--Ueda and Harada--Kaveh \cite{MR2609019,MR3425384}
construct integrable systems from a toric degeneration
of a smooth projective variety
satisfying some additional conditions.
We briefly recall their construction.

Let~$B$ be a manifold containing two points $0$ and $1$ and let~$\pi\colon\mathfrak{X}\to B$\index{p@$\pi\colon\mathfrak{X}\to\bA^1$}
be a flat family of complex projective varieties of dimension~$n$,
% XXX-proj: need morphism to be projective, not just all fibers (or XXX-kahler is sufficient)
% XXX: also affine line on the base is not good, but loc.cit. allow any curve (or a disc)
% XXX: check (cf. with Manon degeneration)
% Pieter: it seems to be fine? The notation A^1 for C is the only thing I can notice being different from what is written there?
such that~$X_0:=\pi^{-1}(0)$ is a toric variety,
whilst the general fiber~$X_t:=\pi^{-1}(t)$ for~$t\neq 0$ is smooth.
Assume moreover that
\begin{itemize}
  \item the singular locus of the total space~$\mathfrak{X}$
        is contained in the singular locus of~$X_0$;
  \item the regular part of~$\mathfrak{X}$ has a K\"ahler form
        which restrict to a torus-invariant K\"ahler form on the regular part~$X_0^\regular$ of~$X_0$.
\end{itemize}

The choice of a piecewise smooth curve~$I\colon[0,1]\to B$ from~$0$ to~$1$ in~$B$
gives
%avoiding the non-zero critical values gives\checkthis{toric degenerations in NNU and HK are isotrivial outside 0},
by symplectic parallel transport along~$I$
a symplectomorphism~$\widetilde{I}\colon X_0^\regular\to X_1^\regular$,
where~$X_1^\regular$ is an open subset of~$X_1$.

The singular toric variety has a natural map
$\Phi_0\colon X_0 \rightarrow \mathbb{R}^n$
which makes it an integrable system.
It is well-known that
the image of $\Phi_0$
is the moment polytope $P$ associated to the toric variety $X_0$.

Using the symplectomorphism~$\widetilde{I}$
we can transport the integrable system
$\Phi_0\colon X_0^\regular\to\bR^n$
to the integrable system
\begin{equation}
\Phi_1 := \Phi_0\circ\widetilde{I}^{-1}\colon X_1^\regular\to\bR^n
\end{equation}
on~$X_1$.
Observe that the same flow gives a~$C^{\infty}$ map
from the compact manifold~$X_1$ to the compact space~$X_0$,
however over the preimage of the singularities,
one cannot extend the torus action.
A suitable choice of the form $\omega$ and path $I$
makes the map real-analytic, cf.~\cite{MR1769452}.
Hence $\Phi_1$ extends as  a map between
$X_1\rightarrow \mathbb{R}^n$
but the functions only Poisson-commute
on an open dense subset of $X_1$.
% Kutzschebauch-Loose, also

Let~$P$ denote the moment polytope of the toric variety~$X_0$. Let
\begin{equation}
  \ell_i(u):=\langle v_i,u\rangle-\tau_i
  \label{equation:ell}
\end{equation}
for~$i\in\{1,\ldots,m\}$ denote the affine equations defining the polytope, i.e.
\begin{equation}
  P:=\Phi_0(X_0)=\{ u\in\bR^n\mid \forall i=1,\ldots,m\colon\ell_i(u)\geq 0 \}.
\end{equation}
The convex hull of the~$-v_i/\tau_i$'s as in~\eqref{equation:ell} is the polar dual~$P^\circ$ of the moment polytope~$P$. The polytope~$P^{\circ}$ will be referred to be as the fan polytope.

Further observe that we have a natural inclusion~$\iota : X_0 \to \mathfrak{X}$ and the map
\begin{equation}\label{eqn:totalintsys}
~c: \mathfrak{X} \to X_0
\end{equation} given by inverse of the symplectic transport on each fiber $X_{t}$. For each $t\in I$, we can consider maps
\begin{equation}
c_t:=c\circ \iota_t: X_t\rightarrow X_0,
\end{equation} where~$\iota_t :X_t\to \mathfrak{X}$ is the natural inclusion. In particular~$\Phi_1= \Phi_0\circ c_1$.

Let $u\in P$ be an interior point of $P$ and let $L_0(u)$ be a Lagrangian torus in $X_0$. By symplectic parallel transport via the map $c_t$, we get a Lagrangian torus
\begin{equation}\label{eqn:lagtori}
L_t(u):=(\Phi_0\circ c_t)^{-1}(L_0(u))\subset X^{\regular}_t.
\end{equation} These tori are all isomorphic for a fixed $u \in P$ and for any $t\in I$.

In \cref{section:toric} we have discussed a toric degeneration of~$\odd$,
and we have now recalled the construction of an integrable system.
We will use the following notation when we apply this construction
with~$X_1$ as~$\odd$.
\begin{definition}
  \label{definition:the-integrable-system}
  Let~$(\graph,c)$ be a trivalent graph with one colored vertex. Then
  \index{P@$\Phi_{\graph,c}$}
  \begin{equation}
    \Phi_{\graph,c}\colon\odd\to\bR^{3g-3}
    \label{equation:nnu-system}
  \end{equation}
  is the map obtained by applying the Nishinou--Nohara--Ueda construction
  to Manon's toric degeneration.
  $\Phi_{\gamma,c}$ gives an integrable system when restricted to $\odd^\regular$.
  Its image will be denoted~$\Delta_{\graph,c}$.
\end{definition}

\subsection{Construction of monotone Lagrangian tori in \texorpdfstring{$\odd$}{the moduli space of rank 2 bundles}}
Now to every trivalent graph~$\graph$ of genus $g$ without leaves
we will associate a monotone Lagrangian torus $L_\graph$
on the moduli space~$\odd$. In \cref{corollary:non-Hamiltonian-isotopic} we will
comment on how the monotone Lagrangian tori for different choices
of~$\graph$ can be compared to each other.

First we recall the definition of Maslov index.
Let $L$ be a Lagrangian submanifold \index{L@$L$} $L$ in
a symplectic manifold $(M,\omega)$,
and let $D$ be a two-dimensional disk $D$ with boundary
a circle $\partial D \cong \sphere^1$.
For a continuous map $f\colon (D,\partial D) \to (M,L)$
its restriction $\partial f\colon \partial D \to L$ is a loop on $L$.
Denote the class of $f$ in the relative homotopy or homology group
as $[D,\partial D] = [D] = [f]  \in \pi_2(M,L) \to \HH_2(L,\bZ)$
and the boundary class as
\begin{equation}
[\partial D] = [\partial f] = \partial[f] \in \pi_1(L) \cong \HH_1(L,\bZ).
\end{equation}
The bundle $\LGr(\mathrm{T}M)$ of Lagrangian Grassmannians
$\LGr(\mathrm{T}_m M) \cong \LGr(\bR^{\dim_\bR M}) =: \LGr$
trivializes when pulled back to the contractible space $D$.
\begin{definition}
  \label{definition:maslov}
  The trivialization of
  the pullback of $\LGr(\mathrm{T}M)$
  to $D$ via $f$ gives the tangential section
  $p\to \mathrm{T}_{f(p)} L \in \LGr(\mathrm{T}M_{f(p)})$
  of $(\partial f)^* \LGr(\mathrm{T}M) \cong \partial D \times \LGr$
  a class in $\pi_1(\LGr) \cong \HH_1(\LGr) \cong \bZ$
  known as the \index{M@$\mu(D)$} \emph{Maslov index}\footnote{This notation hides the mapping $f$ from a pair $(D,\partial D)$ to $(M,L)$.}~$\mu(D)$
	We will further call such maps \emph{disks of Maslov index $\mu$}.

  For an orientable $L$ the lifting of $L$ to the orientable Lagrangian Grassmannian $\LGr^+$
  guarantees that the Maslov index takes even values, i.e.
  $\mu(D) \in \pi_1(\LGr^+) = 2\bZ \subset \bZ = \pi_1(\LGr)$.
\end{definition}

%For the assumptions on the torus~$L\subset X_1$ we r
Recall the following definition.
\begin{definition}
\label{d:monotone}
Let~$L$ be a Lagrangian submanifold of
a symplectic manifold $(M,\omega)$.
We say that~$L$ is (positive) \emph{monotone}
if the symplectic area~$\HH_2(M,L;\bZ)\to\bR$
and the Maslov index~$\HH_2(M,L;\bZ)\to\bZ$
are (positively) proportional to each other.
\end{definition}

We now discuss a general result
which is implicit in the works of
Nishinou--Nohara--Ueda and Ritter
\cite{MR2609019,MR2879356,MR3548462}.

\begin{lemma}
\label{lemma:monotone}
Let~$X_1$ be a simply connected smooth variety of dimension~$n$
admitting a degeneration to a normal toric variety~$X_0$.
For any interior point~$u \in P$,
let~$L=L_1(u)$ be a Lagrangian torus in~$X_1$
constructed via symplectic parallel transport.
Identify $\HH_1(L,\bZ)$ with the cocharacter lattice~$N$ of the toric variety~$X_0$.
Assume that the vertices of~$P^{\circ}$ generate~$N$.
Then there is a isomorphism
\begin{equation}
\HH_2(X_1,L;\bZ) \cong N \oplus \HH_2(X_1,\bZ) .
\end{equation}
\end{lemma}

\begin{proof}
Consider the long exact sequence for the homology of the pair~$(X_1,L)$,
\begin{equation}
\label{equation:ses}
    \HH_2(L,\bZ)
\to \HH_2(X_1,\bZ)
\to \HH_2(X_1,L;\bZ)
\to \HH_1(L,\bZ)
\to \HH_1(X_1,\bZ).
\end{equation}
% Because~$L\cong(\sphere^1)^{3g-3}$ we get that~$\pi_2(L)=0$.
The last term vanishes, as~$X_1$ is simply connected.
Moreover by \cref{lemma:sergeytorictrick},
the image of $\HH_2(L,\bZ)$ in $\HH_2(X_1,\bZ)$ is zero.
\iffalse
Firstly for any~$n$-torus~$\left(\bC^*\right)^n$,
consider the natural embedding
$\iota: \sphere^1 \times \sphere^1 \times {\pt} \rightarrow
        \bA^1 \times \left(\bC^{*}\right)^{n-1}.$
Then the image of~$\HH_2(\sphere^1\times \sphere^1,\bZ)$ under~$\iota$ is zero.

 Now for every ray $r_i$ of $P^{\circ}$, consider the divisor $D_{r_i}\in X_1^{\reg}$. The natural inclusion

% \[ L\rightarrow \cup_{r_i} D_{r_i}\times \bA^{1}\subset X_1^{\reg} \]
% induces a trivial map on second homology.
% Since the rays $r_i$ generate $N$ by assumption, we are done.
\fi
Hence we obtain a short exact sequence from~\eqref{equation:ses},
which allows us to identify
\begin{equation}\label{equation:skeleton}
\begin{aligned}
\HH_2(X_1,L;\bZ)
&\cong\HH_1(L;\bZ)\oplus\HH_2(X_1;\bZ) \\
&\cong N\oplus \HH_2(X_1;\bZ).  \\
\end{aligned}
\end{equation}
\end{proof}

The following lemma is used in the proof of \cref{lemma:monotone},
however it does not require the assumption that $X_1$ is projective.
\begin{lemma}
\label{lemma:sergeytorictrick}
In the situation of  \cref{lemma:monotone}
the map $\HH_2(L,\bZ) \rightarrow \HH_2(X_1,\bZ)$ is zero.
\end{lemma}
\begin{proof}
Let $P^\regular$ be the image of the moment map
from the smooth part~$X_0^\regular$ of the toric variety~$X_0$.
Let~$X_0^\onecosk$ be the toric variety
whose fan is the 1-skeleton of the fan of $X_0$
and let~$P^\onecosk$ be the image of ~$X_0^\onecosk$ under the moment map.
In particular~$P^\onecosk$ is a union of open faces of codimension at most one.
Let $P^\interior$ be the interior of $P$.

% Let $P^\onecosk$ be the~$1$-coskeleton of the polytope~$P$.
Let~$X_0^\interior$ denote the inverse image of~$P^\interior$.
Similarly define $X_1^\interior$.
Both~$X_1^\interior$ and~$X_0^\interior$ contain a torus isomorphic to~$L$.
Consider the following diagram.
\begin{equation}\label{eqn:keyladder}
\begin{tikzcd}
X_1 \arrow[r,] & X_0 \arrow[r, twoheadrightarrow] & P \\
X_1^\regular \arrow[u, hook]\arrow[ r, "\cong"] & X_0^\regular \arrow[r,twoheadrightarrow]\arrow[u, hook] &P^\regular\arrow[u,hook]\\
X_1^\onecosk \arrow[u, hook]\arrow[ r, "\cong"] & X_0^\onecosk \arrow[r,twoheadrightarrow]\arrow[u, hook] &P^\onecosk\arrow[u, hook]\\
X_1^\interior \arrow[u, hook]\arrow[ r, "\cong"] & X_0^\interior \arrow[r,twoheadrightarrow]\arrow[u, hook] &P^\interior\arrow[u, hook]\\
L=L_1(u) \arrow[r,"\cong"] \arrow[u,hook]& L=L_0(u) \arrow[r,twoheadrightarrow]\arrow[u, hook] & u\arrow[u, hook]
\end{tikzcd}
\end{equation}
Since the tori~$L_(u)$ and~$L_0(u)$ are isomorphic by construction, we denote both of them by~$L$ in this section.

By construction $P^\regular = P\backslash P^\singular$,
where $P^\singular$ is the singular locus of $P$.
To show that the image of $\HH_2(L,\bZ)$ in $\HH_2(X_1,\bZ)$ is zero
it would suffice to show that the image of $\HH_2(L,\bZ)$
is zero in $\HH_2(X_0^\regular,\bZ)$
while going up along the middle column of the diagram~\cref{eqn:keyladder}.

Now let $V=\{v_1,\dots, v_m\}$ be the vertices of $P^{\circ}$.
Each $v_i$ corresponds to a unit sphere $\sphere^1$ in~$L$.
The toric variety $T_V$ whose fan is $\bigcup_{v_i\in V}\bR_{\geq 0}v_i$
is homotopy equivalent to $L\cup\bigcup_{v_i \in V} D_{v_i}$,
where $D_{v_i}$ is a disk whose boundary is the one-cycle corresponding to $v_i$.
By our notation $T_V$ is nothing by $X_0^\onecosk$.
We have the following diagram:
\begin{equation}
\begin{tikzcd}
\sphere^1 \arrow[r, hook] \arrow[d, hook]& L  \arrow[d, hook]\\
	    D_{v_i}\arrow[r,hook] & T_{V}.
\end{tikzcd}
\end{equation}

Hence $\HH_1(T_V,\bZ) = L/\bigoplus_{v_i\in V}\bZ v_i$.
But by assumption $V$ generates $L$.
Hence $\HH_1(T_V,\bZ)$ is zero.
Moreover the natural embedding of
\begin{equation}
  \sphere^1\times \sphere^1 \times p \hookrightarrow L \hookrightarrow D_{v_i}\times \left(\bC^\times\right)^{n-1}
\end{equation}
induces the zero map between $\HH_2(\sphere^1\times \sphere^1, \bZ)\rightarrow \HH_2(D_{v_i}\times \left(\bC^\times\right)^{n-1},\bZ)$. Thus we are done.
\end{proof}

\begin{remark}
The toric variety $T_V$ and its homotopy retract to the union of
the Lagrangian torus $L$ with attached holomorphic disks $D_i$
is borrowed from Galkin--Mikhalkin \cite{Galkin-Mikhalkin-hq},
where it is used to compute the (non-abelian) relative second homotopy group
$\pi_2(T_V,L)$ responsible for a ``quantization'' of the Floer potential of surfaces.
\end{remark}

%\begin{remark}
%Monotonicity as it originally appears in the work of \textcite{Floer-symp}
%is a weaker statement, for which proportionality
%of the Maslov index and symplectic area on the spherical part of homology would suffice.
%This can be deduced easier using asphericity of the torus.
%\end{remark}

The following proposition gives us a useful criterion to test
whether the Lagrangian torus~$L_1(u)$ in \cref{lemma:monotone} is monotone.
\begin{proposition}
  \label{proposition:monotonegenerali}
  Let $u\in P$ be an interior point as in \cref{lemma:monotone}
  such that that $\ell_i(u)=\ell>0$,
  where $\ell_i$ are the length functions~\eqref{equation:ell} defining the polytope $P$.
  Further assume that~$(X_1,\omega)$ is monotone and the degeneration preserves second Betti numbers.
  Then the Lagrangian torus~$L_1(u)$ is monotone.
\end{proposition}

\begin{proof}
% Let~$D$ be a class of a holomorphic disc
% in~$\HH_2(X_1,\bZ)\subset \HH_2(X_1,L;\bZ)$.
% Its area $\int_D \omega$ is positive,
% so monotonicity implies positivity of Maslov index.
%% XXX: Four lines above are misplaced
Consider the (holomorphic) disks $D_i$ of Maslov index~$2$ (\cref{equation:boundary} and \cref{remark:second-betti})
enumerated by the vertices~$v_i$ of the polytope~$P^\circ$.
By \cite[Theorem 10.1]{MR2609019} or \cite[Theorem 8.1]{MR2282365}
they have symplectic area
\begin{equation}
  \int_{D_i} \omega = \ell_i(u) = \ell > 0.
\end{equation}
Non-negative linear combinations of the classes
of their boundaries $[\partial D_i] \in \HH_1(L,\bZ) = N$
generate a finite index sublattice.
Hence any element $\beta \in \HH_2(X_1,L;\bQ)$
can be presented as $\beta = \beta_1 + \sum b_i [D_i]$,
a linear combination of $[D_i]$
and some $\beta_1 \in \HH_2(X_1,\bQ)$.
Moreover there exists $a_i\geq 0$ such that $\sum a_i > 0$
such that $\sum a_i [D_i] \in \HH_2(X_1,\bZ)$.

Then we have that
\begin{equation}
  \int_\beta \omega = \left( \sum b_i \right) \ell + \int_{\beta_1} \omega,
\end{equation}
and
\begin{equation}
  \int_{\sum a_i [D_i]} \omega = \left( \sum a_i \right) \ell > 0
\end{equation}
where~$\mu(\sum a_i [D_i]) = 2\sum a_i > 0$. Monotonicity of $(X_1,\omega)$ implies that
\begin{equation}
  \int_{\beta_1} \omega
  =\mu(\beta_1)  \frac{\int_{\sum a_i [D_i]} \omega}{\mu(\sum a_i [D_i])}
  =\mu(\beta_1)/2 \cdot \ell,
\end{equation}
hence
$\int_\beta \omega = (\ell/2) \mu(\beta)$.
\end{proof}

Now we apply \cref{proposition:monotonegenerali}
to produce a Lagrangian torus in~$\odd$, using Manon's toric degeneration and the integrable system described in \cref{definition:the-integrable-system}.
% If $(\graph,c)$ be a maximally connected trivalent graph with one colored vertex,
% then the central fiber is terminal by \cref{theorem:terminal-criterion},
% and in \cref{lemma:polytope-description} we have obtained a description
% of the lattice points of the fan polytope~$P_\graph^\circ$,
% as the origin and the vertices~$p(v,s)$.
% This description allows us to prove the following.

\begin{proposition}
  \label{proposition:monotone-torus}
  %% Assumptions not needed
  % Let~$(\graph,c)$ be a maximally connected trivalent graph with one colored vertex.
  % Assuming \cref{conjecture:small-resolution},
  The Lagrangian torus~$L:=\Phi_{\graph,c}^{-1}(0)\subset\odd$ is monotone.
\end{proposition}

\begin{proof}
  For each vertex~$v$ of the graph~$(\graph,c)$
  and $s\in C_{v,c}$~\eqref{equation:colored-cycles},
  consider the equations~$\ell_{v,s}$ from~\eqref{equation:colored-inequalities}
  used in \cref{subsection:manon}
  to describe the polytope~$P_{\graph,c}$ and the polar dual~$P^{\circ}_{\graph,c}$.
  Rewrite~$\ell_{v,s}$ as in~\eqref{equation:ell}
  \begin{equation}
    \langle p(v,s),u \rangle - \tau_{v,s}
  \end{equation}
  where~$p(v,s)$ is as in~\eqref{equation:vertices}
  and $\tau_{v,s}=-1$.
  Setting~$u=0$ we obtain that~$\ell_{v,s}(0)=1$.
  Observe that the lattice points $p(v,s)$
  of the polytope~$P_{\graph,c}^\circ$,
  are non-zero,
  that the vertices form a subset by \cref{theorem:terminal-criterion},
  and furthermore span the lattice $N_{\graph}$.
  The statement then follows directly from \cref{proposition:monotonegenerali}.
\end{proof}

There is a combinatorial non-abelian Torelli theorem for these toric degenerations,
see \cite{gp-torelli}.
This allows one to reconstruct the trivalent graph from the degeneration,
and thus we can obtain the following corollary.

\begin{corollary}
  \label{corollary:non-Hamiltonian-isotopic}
  Let~$\graph$ and~$\graph'$ be different trivalent graphs of genus~$g$.
  Let~$L_\graph$ and~$L_{\graph'}$ be the monotone Lagrangian tori in~$\odd$ constructed above.
  Then they are not Hamiltonian isotopic to each other.
\end{corollary}
%This is similar to the construction of many non-Hamiltonian isotopic
%monotone Lagrangian tori in del Pezzo surfaces \cite{MR3509972,MR3663599}.

\begin{remark}
  \label{remark:even-symplectic}
  In the case of~$\even$ symplectic transport gives
  a symplectomorphism from
  the non-singular locus of the toric fiber
  to an open subset in non-singular locus of~$\even$.
  The corresponding Nishinou--Nohara--Ueda integrable system
  seems to be quite different from Goldman--Jeffrey--Weitsman
  integrable system from \cite{MR1204322}, despite sharing the base~$P_\graph$.
  In the former case the preimage of the interior of the polytope
  is the cotangent bundle of the torus and in the latter it is singular.
\end{remark}

\paragraph{On the nature of mutations}
We now briefly speculate on the relationship between the different Lagrangian tori $L_\graph$,
constructed for different choices of trivalent graphs,
i.e.~what role of elementary transformations from \cite[\S2.2]{gp-tqft} play.
In the context of
Batyrev's small toric degeneration approach to Fano mirror symmetry
cluster mutations in the sense of Berenstein--Fomin--Zelevinsky \cite{MR2110627}
and Fock--Goncharov \cite{MR2567745}
were rediscovered
independently
\begin{itemize}
  \item for spherical varieties using representation theory, by Rusinko \cite{MR2421321},
  \item for del Pezzo surfaces and Fano threefolds using deformation theory, by Galkin \cite{Galkin-phd}.
\end{itemize}
It was observed that the Laurent polynomials associated by a mirror construction to
different toric degenerations of a Fano manifold
% (e.g. of the projective plane and other del Pezzo surfaces,
% or of the Grassmannian)
are related to each other by volume-preserving birational changes of coordinates.

An algebraic theory of mutations of mirror potentials was constructed
in joint works of Galkin with Usnich and Cruz Morales \cite{MR3033547,Galkin-Usnich},
and the respective excessive Laurent phenomenon was proved in this context:
\begin{itemize}
  \item in \cite{Galkin-Usnich} from a more algebro-geometric point of view
    (similar to that of Gross--Hacking--Keel--(Kontsevich) \cite{MR3350154,MR3758151}),
  \item and then in \cite{MR3033547}
    using the relation to cluster algebras and their upper bounds.
\end{itemize}
In this way a framework free from (commutative) algebraic geometry was constructed,
with a view towards quantization.
Amongst others, the notions of exchange collections, upper bounds,
and mutation rules for them are introduced,
and the Laurent phenomenon was proved algebraically and geometrically.

Ten equivalence classes of exchange collections in dimension~2
were constructed in \cite{Galkin-Usnich,MR3033547}.
It was conjectured that each respective upper bound written in all possible seeds
is generated by a single Laurent polynomial with positive integer coefficients,
which is equal to a Floer potential for a monotone Lagrangian torus
in an integrable system associated by the Nishinou--Nohara--Ueda construction
to the respective toric degeneration of a smooth del Pezzo surface.
Moreover it was conjectured that these birational transformations coincide
with Lagrangian wall-crossing formulae, and that Floer potentials
for all embedded monotone Lagrangian tori are equal to one of the constructed functions.

Most of these were settled positively by Vianna, Galkin--Mikhalkin,
Pascaleff--Tonkonog and Ekholm--Rizell--Tonkonog
\cite{MR3509972,Galkin-Mikhalkin-2,1711.03209,1806.03722,Galkin-Mikhalkin}.
However in higher dimensions
(apart from the holomorphic symplectic case) our knowledge is much less complete.
Already at the most basic level:
for mirror symmetry of del Pezzo surfaces Cruz Morles, Galkin and Usnich give a recursive procedure
to construct infinitely many different Laurent polynomials
from just one was constructed,
but we do not know any general higher-dimensional recursive rule of generalized ``mutations''
that would enjoy a Laurent phenomenon
and the same time produce enough potentials.
% ACGK for 3-folds, comb. mutations and computer experiments

Pascaleff--Tonkonog \cite{1711.03209}
introduced the idea that
the birational mutation of potentials of Cruz Moralez--Galkin \cite{MR3033547}
(which itself is a geometric lift in the sense of Berenstein--Zelevinsky
of a piecewise-linear homeomorphism of polytopes)
can be lifted further to a \emph{Lagrangian mutation},
i.e.~a symplectic surgery
between the respective Lagrangian tori
with the help of auxiliary Lagrangian disc,
such that class of its boundary belongs to the exchange collection.

Dimitroglou--Eckholm--Tonkonog \cite{1806.03722}
provide further insight into mutations of two-dimensional Lagrangian tori
by interpreting it as two deformations of an immersed Lagrangian sphere.
For the latter they propose
 one unifying Floer-like potential
  with values in a (non-commutative)
   multiplicative preprojective algebra,
 out of which
  the Floer potentials for Lagrangian tori
   are obtained by specialization.
Finally, Hong--Kim--Lau \cite{1805.11738} also study study Floer potentials
for immersed Lagrangian subvarieties,
which are isomorphic to a product of an embedded torus with an immersed two-sphere,
from the perspective of Abouzaid--Auroux--Katzarkov \cite{MR3502098}, % Abouzaid--Auroux--Katzarkov
based on Strominger--Yau--Zaslow \cite{MR1429831}.  % Strominger--Yau--Zaslow

This leads us to the following observation.
\begin{remark}
  \label{remark:speculation}
  One naturally expects that the rational changes of coordinates
  corresponding to elementary transformations of graph potentials
  are similarly associated with a (new kind of) Lagrangian mutations
  between the monotone Lagrangian tori $L_\graph$ in $\odd$.
\end{remark}

%S., [21 Feb 2022 at 17:53:55]:
%Well, classically it is
%
%x x' = M + M'
%
%M and M' are monomials of n-1 other variables,
%these n-1 with either x or x' make a collection of n variables in two neighbouring seeds
%
%On the left hand side you have a monomial of degree 1,1, on the right hand side a binomial.
%If you move all them to one side you have a trinomial which equals to zero.
%And e.g. Plücker relations for Grassmannian Gr(2,n) are trinomials, and if you do inverse operation and move the negative one to the left you have relation of this type.
%
%In contrast our relations are of form
%
%one binomial equals another binomial,
%
%there is one more monomial - totally it is 4-nomial, not 3-nomial.

One new and peculiar feature
is that elementary transformations of graph potentials naturally
equate a sum of two monomials % calling this a binomial is linguistically correct, but I don't recall seeing this (often)
to a sum of two monomials
as in \cite[Equation (43)]{gp-tqft}.
This is in contrast to (various generalizations of) cluster algebras
where a single mutation involves
an equality between a monomial and a sum of two monomials.
A limiting procedure, discussed in \cite[Remark~4.19]{gp-tqft}
specializes these new transformations
to the transformations of Nohara--Ueda \cite{MR3211821,MR3391735,MR4118149}
between Floer potentials of Euclidean polygon spaces.
% Yuichi Nohara, Kazushi Ueda
% [1111.4809]  Toric degenerations of integrable systems on Grassmannians and polygon spaces
% [1405.1058]  Goldman systems and bending systems
% [1711.04456] Potential functions on Grassmannians of planes and cluster transformations
In \cite{MR4118149} they show how
the specialized transformations produce a covering for
a total space of Rietsch's Lie-theoretic mirror potential from \cite{MR4034891}
for the Grassmannian of planes,
% Konstanze Rietsch and Lauren Williams,
% [1507.07817] Cluster duality and mirror symmetry for Grassmannians,
which generally is an open subspace in a Langlands dual homogeneous variety.
In fact, Nohara--Ueda \cite{MR4118149} show that formulae
for mutations of their potentials are
\emph{equivalent}
to Plücker relations of the Grassmannian,
so in some sense \cite[Equation (43)]{gp-tqft} is related to Plücker relations
as the group $SU(2)$ is related to Lie algebra $su(2)$.
It would be interesting to understand this behavior better.

%% file: sympl-disks.tex
\section{Disks of Maslov index two and disk potentials}
\label{section:potentials}
We will now introduce the disk potential,
for a Lagrangian torus inside a symplectic manifold~$(X,\omega)$. % and recall some of the necessary ingredients.
 %\begin{example}
%\label{example:maslov}
%Let $D$ be the unit disk $D = \{|z|<1\} \subset \bC$.
%For $k=0,1,2,\dots$ the map $w = f(z) = z^k$ from $(D,\partial D)$ to $(\bC,\{|w|=1\})$
%has Maslov index $2k$, and $f(z) = \bar{z}^k$ has Maslov index $-2k$.
%
%Similarly, for any point $u_i\in\bR$ and $\phi_i\in\bR/(2\pi)$
%the map
%\begin{equation}
%  z \mapsto (z^k \cdot \exp(u_1+\ii\phi_1),\exp(u_2+\ii\phi_2),\dots,\exp(u_N+\ii\phi_N))
%  = (w_1,w_2,\dots,w_N)
%\end{equation}
%from $(D,\partial D)$ to $(M=\bC\times(\bC^*)^{N-1}, L=\{|w_i|=\exp u_i\})$
%has Maslov index $2k$ (for the symplectic structure on the target equivalent to the product)
%and sends the $k$ points of the form $z = \exp 2\pi\ii j/k \in \partial D$ to
%the point
%$P = (\exp(u_1+\ii\phi_1)\dots,\exp(u_N+\ii\phi_N)) \in (\bC^*)^N$.
%
%More generally, let~$X$ be a normal toric variety. For any point $P \in (\bC^*)^N \subset X$ on the open orbit of~$X$
%and for any irreducible torus-invariant divisor $Z\subset X$,
%there is an open subvariety $X_Z \subset X$
%isomorphic to $\bC \times (\bC^*)^{N-1}$ such that restriction of $Z$ to $X_Z$
%is given by equation the $w_1=0$ in
%a system of coordinates $w_1 \in \bC, w_2, \dots, w_N \in \bC^*$ on $X_Z$.
%The previous example gives a disk of Maslov index $2$ on $X_Z \subset X$
%with boundary on the fiber $L = \Phi^{-1}\circ \Phi(P)$
%of the moment mapping $\Phi\colon X \to \Delta$.
%In this case the preimage of $P\in L$ is the unique point $z=1\in D$.
%\end{example}
The next ingredient
is the choice of an almost complex structure~$J$. We say that an almost complex structure~$J$
is \emph{$\omega$-tame}
if $\omega(v,J v)>0$ for all non-zero tangent vectors $v$.
%is positive-definite (Riemannian metric on $M$).
For a fixed $\omega$ the space of $\omega$-tame almost complex structures
is contractible, hence the Chern classes of $J$ are symplectic invariants.
In this situation the Maslov index of a $J$-holomorphic map $f$
equals twice its first Chern number.
In particular a $J$-holomorphic disk with boundary on $L$
that intersects an effective anti-canonical divisor transversally
in a unique point has Maslov index $2$.
% XXX - whe L is disjoint from Y and D is irreducible J-holomorphic,
% the intersection is non-negative.
More generally, one can show that the Maslov index controls
the dimension of the (virtual) fundamental cycle
on the (Kuranishi) moduli space of $J$-holomorphic maps
\begin{equation}
f\colon(C\ni P_1,\dots,P_a;\partial C\ni Q_1,\dots,Q_b) \to (M,L)
\end{equation}
from a Riemann surface $C$ with $a$ marked points
in its interior and $b$ marked points on the boundary $\partial C$.
% XXX: specify dimension formula
% \dim_\bR [\moduli_{g,n,b,m}(M,L;β)] = μ(β) + ...

\begin{definition}
  For a~$J$ fixed the \emph{disk/Floer potential} of~$L$,
  denoted~$m_0(L)$ (or~$\mathfrak{PO}$ or $W_L$\index{W@$W_L$})
  is the generating function
  for the $J$-holomorphic disks~$(D,\partial D,p)\subset(M,L,P)$ of Maslov index~2,
  whose boundary passes through a specified point~$P$ on~$L$
  (one can also fix an interior point and ask it to lie
  on a fixed effective anti-canonical divisor).
\end{definition}
In general, it is defined on
the moduli space of weakly bounded cochains as
\begin{equation}
m_0(L)(x) :=
\hspace{-2em}
\sum_{\substack{(D\supset\partial D\ni p)\subset (M\supset L\ni P) \\ \mu(D)=2}}
\hspace{-2em}\exp(\langle[\partial D],x\rangle) T^{\int_D \omega}
\end{equation}
where
\begin{itemize}
  \item $\int_{\ast}\omega\colon\HH_2(M,L;\bZ)\to\bR$ is the symplectic area;
  \item $\mu\colon\HH_2(M,L;\bZ)\to\bZ$ is the Maslov index;
  \item $x\in\HH^1(L,\Lambda_0)$ is a cohomology class valued in
  Novikov ring $\Lambda_0$;
  \item $\Lambda := \left\{ \sum_{i=0}^{+\infty} a_i T^{\lambda_i}
                 \mid a_i\in\bQ,\lambda_i\in\bR,
                 \lim_{i\to+\infty}\lambda_i=+\infty \right\}$
        and $\Lambda_0$ \index{L@$\Lambda_0$}
      is its subring of elements with non-negative valuation
\[
\mathfrak{v}\left( \sum a_i T^{\lambda_i} \right) := \min\left\{ \lambda_i\,|\, a_i\neq 0 \right\}
\in \bR
\]
\end{itemize}
The Novikov ring is used to deal with symplectic areas
$T^{\int_{\Sigma} \omega}$
when the class $[\omega]\in\HH^2(M,\bR)$ of the symplectic
 form is not an integer,
or when the surface $\Sigma$ is not closed.
However under assumptions of integrality and monotonicity,
%(which will be introduced and established later)
the generating functions can be rewritten as Taylor series.
For similar results on monotone toric manifolds
and the passage between critical points
in Archimedean and Novikov contexts,
see the work of Ritter \cite[Section~7.9]{MR3548462}
and Judd--Rietsch \cite{1911.04463}.

% Judd--Rietsch

For generic $\omega$-tame $J$
the potential gives a well-defined
element in the Novikov ring
thanks to the Gromov compactness theorem.

We now discuss several methods to compute the
disk potential for the moduli space $\odd$ with
respect to a monotone Lagrangian toric $L$ obtained
via symplectic parallel transport as in~\cref{definition:the-integrable-system}. We discuss the
general set-up of toric degenerations and analyse the
various special cases starting with the existence of
small resolutions that connect
to~\cref{conjecture:small-resolution} and then
the case of terminal singularities and ending
with the most
general case.
These results improve upon the results of Nishinou--Nohara--Ueda \cite{MR2879356}.

\paragraph{The special case of small resolutions}
The symplectic parallel transport for integrable systems
is used to show the following \cite[Theorem~1]{MR2879356}.
\begin{theorem}[Nishinou--Nohara--Ueda]
\label{theorem:nnu}
Consider a toric degeneration~$\mathfrak{X}\to B$
as in \cref{subsection:integrable-system} and
let $\Phi:=\Phi_1: X_1^{\regular}\rightarrow \mathbb{R}^n$
 be the integrable system obtained as in \cref{subsection:integrable-system}.
Assume further that~$X_0$ is Fano, and admits a
small resolution of singularities.

Then for any~$u\in P^{\mathrm{\interior}}\subseteq\bR^n$
in the interior of the moment polytope,
the Floer potential of the Lagrangian torus fiber~$L(u):=\Phi^{-1}(u)$
equals the disk potential and
is given by\index{P@$\mathfrak{PO}$} \begin{equation}
m_0(L)(x) = \sum_{i=1}^m \exp(\langle v_i,x \rangle) T^{\ell_i(u)},
\end{equation}
where~$v_1,\dots, v_m$ are the generators of the polar dual~$P^\circ$
and~$\ell_i$'s are the equations defining
the facets of the moment polytope~$P$
and valuations of $x$ belong to the moment polytope $P$.
\end{theorem}

In particular it says that
if the degenerating toric Fano variety
admits a small resolution of singularities,
then the holomorphic disks with Lagrangian boundary
that contribute to the Floer potential of the smoothing
are transversal to toric strata.

\iffalse
In contrast, for degenerations with transversal
$\mathrm{A}_1$
 singularities there  an additional contribution
  from a disk that intersects it,
in our situation this can be observed for e.g.~on
graphs with loops.
\fi

\begin{remark}
  By~\cref{conjecture:small-resolution},
  it is expected that we can always find a graph $\gamma$,
  such that
  the corresponding toric degeneration $\manontoric{\graph,c}$ of $\odd$ has a small resolution.

  Given a small resolution \cref{theorem:nnu}
  can be applied to relate graph potentials
  to the Floer potential of~$\odd$.
\end{remark}

\paragraph{A generalization of Nishinou--Nohara--Ueda without small resolution}
Lacking a proof of \cref{conjecture:small-resolution}
we will rather generalise \cref{theorem:nnu}
to \cref{theorem:gennnu}

\paragraph{Some curves in~$\pi_2(X_0,L_0)$}

Before we discuss~\cref{theorem:gennnu},
let us construct a system of curves in $\pi_2(X_0,L_0)$.
We will later show that these curves form a
generating set for $\pi_2(X_0,L_0)$.
We denote~$L_0:=L_0(u)$.

For each
irreducible component $D_i$ of the torus-invariant
divisor $D$ on $X_0$
consider the complement of $X$ in the union of all
other irreducible components of $D$,
this is an affine variety isomorphic to
$\mathbb{C}\times (\mathbb{C}^*)^{n-1}$ with
coordinates $x_1;x_2,\ldots,x_n$
and $L_0(u)$ is given by~$|x_k|=u_k$,
where~$u_k$ are the components of~$u$.

For any point $P$ in $L_0(u)$ with coordinates
$z_1,\ldots,z_n$
there is a disk $C_i = C_i(u,P)$ given by a system
of $n-1$ equalities $x_k = P_k$, $k=2,\ldots,n-1$,
and one inequality $|x_1| \leq u_1$. Also the
fibers of $c$ over $c(D_i^o)$ are
$n-1$-dimensional tori, where $c$ is the map in~\eqref{eqn:totalintsys}
and $D_i^0=D_i\backslash \cup_{j\neq i}D_i$. There is
a map from generic $n$-dimensional
torus $L_0=L_0(u)$ to these tori that collapses
one direction corresponding to the boundary
of the disk $C_i$: in the local coordinates from above
just vary $u_1$ to $0$.

So for each $i$ the homotopy class
$[C_i] \in \pi_2(X_0,L_0)$ has the boundary
\begin{equation}\label{equation:boundary}
\partial[C_i] = [\partial C_i] \in \pi_1(L_0) = N
\end{equation}
%(the lattice where the fan lives, dual to monomials $M$)
and under the last identification
$\partial[C_i] = \rho_i$, where $\rho_i \in N$
is the primitive ray corresponding to the divisor $D_i$. 
We will use these curve classes in~\cref{proposition:magiclemmawithsalt}
which is a key step in the proof of the following main result. 
% We have the following proposition that 

%Now~\cref{proposition:magiclemmawithsalt} in the set-up of Gorenstein
% toric degeneration gives the following theorem:
\begin{theorem}
  \label{theorem:gennnu}
  Let $X_1$ be a smooth projective Fano variety
  and let $\mathfrak{X} \rightarrow B$ be a degeneration
  of $X_1$ into a Gorenstein Fano toric variety $X_0$
  such that $X_{t}$ is smooth for $t\neq 0$.
  Moreover assume that the degeneration preserves the second Betti number.
  Let $L_t(u)$ be
  a monotone Lagrangian torus obtained by
  the symplectic parallel transport of
  a Lagrangian torus $L_0(u)$
  in $X_0^\regular$
  for some point $u\in P^{\interior}$.

  Then the Newton polytope of the disk potential~$m_0(L_0(u))$
  equals the fan polytope $P^\circ$ of the toric variety~$X_0$.
  Moreover the coefficients of $m_0(L_0(u))$ corresponding to the rays of $P^{\circ}$ are one.
\end{theorem}
\begin{remark}
	We are removing the need for a small resolution but with an assumption
	that the second Betti number is preserved by the degeneration.
	This assumption is enough for our application as Manon's toric degeneration
	preserve the second Betti numbers.
\end{remark}

\begin{remark}
  \label{remark:second-betti}
The condition on the preservation of the second Betti number in \cref{theorem:gennnu}
can be substituted by the more general notion of $\mathbb{Q}\Gamma$-smoothing
  due to Galkin--Mikhalkin \cite{Galkin-Mikhalkin}.
We refer the reader to \cite[Theorem 1]{Galkin-Mikhalkin} for a precise result.
In particular, for us it is enough to assume that the canonical class $\mathrm{K}_{X_t}\in\HH^2(X_t,\mathbb{R})$
  vanishes on the kernel of the map $\HH_2(X_t,\mathbb{R})\rightarrow\HH_2(X_0, \mathbb{R})$.
These notions are only used in defining the Maslov index and monotonicity for singular symplectic spaces \cite[Section 1.2]{Galkin-Mikhalkin}.
The condition on the preservation of Betti numbers in the statement of \cref{theorem:magic-lemma}
  already holds in the setup of Nishinou--Nohara--Ueda \cite{MR2609019}. 
% FIXME: Explain why. I do not know a proof of this statement for dim > 3!
\end{remark}

Now consider the special case where the only non-zero lattice point of $P^\circ$ are the rays (i.e. $X_0$ is terminal c.f. \cref{sec:terminal}). Then \cref{theorem:gennnu} directly implies
\begin{corollary}\label{corollary-terminal}
%The later is also the Newton polytope of the following Laurent polynomial
  \begin{equation}
   m_0(L_0(u))=\sum_{i=1}^m \exp(\langle v_i,x \rangle) T^{\ell_i(u)},
 \end{equation}
  where~$v_i$ are the generators of the polar dual~$P^\circ$
 and~$\ell_i$'s are the equations defining
  the facets of the moment polytope~$P$
  and valuations of $x$ belong to the moment polytope $P$.
\end{corollary}

We postpone the proof of \cref{theorem:gennnu} and
focus on applications of the theorem for $X=\odd$.
\Cref{corollary-terminal} combined with
Manon's construction of a toric degeneration
has the following important consequence in our setting
that computes the disk potential for $\odd$.
%The proof follows

\begin{theorem}
  \label{theorem:graphisdisc}
  Let $(\graph,c)$ be a trivalent graph with
  one colored vertex. Then the following are equal:
  \begin{itemize}
    \item the graph potential~${W}_{\graph,c}$,
    \item the disk/Floer potential of the Lagrangian
    torus~$L_{\graph,c} \subset \odd$,
  \end{itemize}
  where $L_{\graph, c} = \Phi_{\graph,c}^{-1}(0)$ is
  a fiber of the integrable system
  $\Phi_{\graph,c}\colon\odd^\regular \to P_{\graph,c}$
  associated with Manon's toric degeneration of
  $\odd$ to $\manontoric{\graph,c}$.
  Moreover the monotone torus $L_{\gamma,c}$ is optimal i.e.
  \begin{equation}
    T_{\mathrm{con}}(L_{\gamma,c})=8g-8=T_{\odd}.
  \end{equation}
\end{theorem}
\begin{proof}
  We apply~\cref{theorem:gennnu} to Manon's toric degenerations associated to a trivalent graph
  and observe that for $u=0$ the Lagrangian torus $L_1(0)$ is monotone by \cref{lemma:monotone}.
  Now we consider the following two cases and without loss of generality we assume that the genus of the graph is at least 3.

  \paragraph{Case I}
  Assume that the graph $\gamma$ has no separating edges.
  This assumption combined with \cref{theorem:terminal-criterion}
  implies that the lattice points of the polytope $P^{\circ}_{\gamma,c}$
  come only from the rays and hence by \cref{theorem:gennnu},
  we get the required result from \Cref{corollary-terminal}.
  %\paragraph{Step II} We now consider the general case. Let $Z_{\gamma,c}:=m_0(L_{\gamma,c})-W_{\gamma,c}$
  %be the hypthetical difference. We will be done if we show that
  %$Z_{\gamma,c}=0$. ....

  \paragraph{Case II}
  We now consider the case of a general graph.
  First we observe that for any trivalent graph $(\gamma,c)$,
  the classical periods of $m_{0}(L_{\gamma,c})$ are equal. 
  On the other hand the classical periods of the graph potentials $W_{\gamma,c}$
  for trivalent graphs are also equal by \cref{theorem:graphtqft}.
  Thus we know that the classical periods of $m_0(L_{\gamma,c})$ equal the classical periods of $W_{\gamma,c}$.
  Now by \cref{theorem:gennnu},
  we get that the in order to determine $m_0(L_{\gamma,c})$,
  we need to only focus on the lattice points on the polytope $P^{\circ}_{\gamma,c}$ which do not come from rays.
  By \cref{lemma:grandcentral}, those are exactly the lattice points $\pm e_b^{\vee}$, where $b$ is a bridge.

  Consider the difference $Z_{\gamma,c}:=m_0(L_{\gamma,c})-W_{\gamma,c}$,
  which is supported on the monomials associated to lattice points $\pm e_b^{\vee}$ where $b$ is a bridge.
  Let $z_b^{\pm}$ be the coefficients of the monomials associated to the lattice points $\pm e_b$.
  Now we have equations
  \begin{equation}
    \label{eqn:stupid}
    [(W_{\gamma,c}+Z_{\gamma,c})^m]_0=[W_{\gamma,c}^m]_0 \ \text{for all}\  m\geq 0.
  \end{equation}

Thus we have infinitely many equations in as many variables as there are bridges in the graph.  Thus  it follows that $Z_{\graph,c}=0$ is the unique solution.
Thus we are done with determining the Floer/disk potentials.
\iffalse
  For $m=2$, by \cref{proposition:p2} we know that $[W_{\gamma,c}^m]_0=0$,
  and this in turn forces $z_b^{\pm}=0$.
  Thus we have infinitely many equations in as many variables as there are bridges in the graph.
  Now bridge variables non-adjacent to a loop cannot contribute to the fourth classical period.
  But bridges adjacent to a loop do contribute to the fourth classical period.
  Thus to match up the fourth period of $m_0(L_{\graph,c})$ with that of $W_{\graph,c}$,
  for every bridge $b$ separating a loop,
  we need to assign coefficients $z_{b}^{\pm}$ to the monomials associated to these lattice points $\pm e_b$
  with the property $z_b^{+}z_b^{-}=0$.
  If for any edge $b$ that separates a loop both $z_b^{\pm}<2$,
  then some other coefficient associated to a bridge separating a loop has to be greater than 2.
  This will imply that the $8$th classical periods of $m_0(L_{\graph,c})$ and $W_{\gamma,c}$ will be different
  which contradicts \cref{eqn:stupid}.
\fi
%  Similarly using \cref{eqn:stupid} for various $n$'s

The graph potential
  evaluated at the unique Morse point $(1,\dots, 1)$ in the domain $\mathbb{R}_{+}^{\dim L_{\gamma,c}}$
  gives the value $8g-8$ which agrees with that of $T_{\odd}$ as computed in \cite{MR1695800}. Thus optimality follows. 
\end{proof}
Recall that for each irreducible components $D_i$ of the torus invariant divisor $D$ of $X_0$, we have constructed curves $(C_i,\partial C_i)$ such that $[\partial C]$ are the primitive rays $\rho_i \in N$ corresponding to the divisor $D_i$. We have the following result that implies~\cref{theorem:magic-lemma} 
\begin{proposition}\label{proposition:magiclemmawithsalt}
	If $(C,\partial C)$ is the class of an irreducible curve of Maslov index two,
	then it lies in the convex hull of the classes $[C_i]$.
	Moreover, the class $(C,\partial C)$ (under the identification $[\partial C] \in \pi_1(L_0)=N$) is at the vertex (which is the extremal case) of the fan polytope,
	only if the limit curve lies inside one of these charts of $X_0$ given by one ray.
\end{proposition}

\paragraph{Proof of~\cref{proposition:magiclemmawithsalt}}
We will prove \cref{proposition:magiclemmawithsalt}
first in the special case of terminal singularities,
which is in fact enough for our intended application
involving graph potentials,
and then we will give the general case which
 is of independent interest.
By~\cref{theorem:terminal-criterion}, we know
that the toric varieties associated
to trivalent graphs constructed by Manon have
 Gorenstein terminal singularities.

\paragraph{The special case of terminal singularities}
%We first discuss a proof of~\cref{theorem:gennnu}
%assuming that the toric variety $X_0$ has Gorenstein terminal singularities.
Let~$X_0$ be a $\mathbb{Q}$-Gorenstein,
terminal,
Fano toric variety
with~$D=\sum_{i\in S}{D_i}$ the boundary
anti-canonical divisor
(the complement of the open orbit)
and let~$C$ be an irreducible curve not
contained in~$D$.
Let~$P$ be a point such that the local
intersection index of $C$ and $D$ in $P$ equals 1.
%Observe that in this terminal case,
%we do not use the assumption that the degeneration preserves the second Betti numbers.

We have the following lemma.
\begin{lemma}
  \label{lem:pointcount}
  $\#\{i \in S\ | \ P\in D_i\}|=1$.
\end{lemma}

\begin{proof}
  Let $f\colon\widetilde{X}_0 \rightarrow X_0$ be
  any toric resolution of singularities
  and consider the unique irreducible
  curve $C'$ such that $f$ maps to $C$.
  Thus $f_*(C')=C$ as~$1$-cycles.
  On the other hand
  \begin{equation}
    f^*(-\mathrm{K}_{X_0})=-\mathrm{K}_{\widetilde{X}_0}+\sum_{j}a_jE_j,
  \end{equation}
  where $E_j$ are exceptional divisors
  and since $\widetilde{X}_0$ is toric, we get
  \begin{equation}
    f^*D=\sum_i D_i' +\sum_j (1+a_j)E_j,
  \end{equation}
  where $D_i'$ denotes the proper preimages of $D_i$.
  Now by the projection formula we get
  \begin{equation}
    \label{equation:obvious1}
    C\cdot D=C'\cdot f^*D.
  \end{equation}
  By assumption the left-hand side of~\eqref{equation:obvious1} is one.
  Expanding the right-hand side of~\eqref{equation:obvious1} gives
  \begin{equation}
    \label{equation:obvious2}
    \sum_i (C'\cdot D'_i) + \sum_j (1+a_j)(C'\cdot E_j)=1
  \end{equation}
  Observe that since $\widetilde{X}_0$ is smooth
  and the divisors $D_i'$ and $E_j$'s are all Cartier divisors,
  the intersection numbers are all integers.
  Moreover all the intersection numbers are non-negative.
  Thus for~\eqref{equation:obvious2} to hold we either have that
  \begin{itemize}
    \item there exists exactly one $i$ such that $C\cdot D_i'=1$
      and all other terms are zero,
      which forces $C$ to be disjoint from $D_k$ for $k\neq i$
    \item there exists an index $j$ such that $C'\cdot E_j\neq 0$
      and thus $C'\cdot f^*D\geq 1+a_j$.
  \end{itemize}
  Now since $X$ has terminal singularities we have that~$1+a_j>1$.
  Thus the second possibility cannot happen.
  Thus we are done.
\end{proof}

%We use Lemma \ref{lem:pointcount} below to prove~\cref{theorem:gennnu}.
\begin{proof}[Proof of~\cref{proposition:magiclemmawithsalt} in the terminal case]
  Let $u$ be a interior point of $P^{\interior}$ as in the statement
  of~\cref{theorem:gennnu} and let $L_0(u)$ be a monotone
  Lagrangian torus in $X_0$. Moreover, we have a family of
  pairs $(X_t,L_t)_{t\neq 0}$ of smooth Fano varieties
  along with monotone
   Lagrangian tori~$L_t=L_t(u) \subset X_t$ as in~\eqref{eqn:lagtori}.
Recall that $B$ has two special points $0$ and $1$ and $X_1=X$.
We further assume that $X_0$ is a $\mathbb{Q}$-Gorenstein,
terminal, Fano toric variety.

Let $u$ be  an interior point in the
moment polytope and $L_1(u)$ be
a monotone Lagrangian torus in $X_0^\regular$. We denote its
symplectic  transport to $X_t$ by  $L_t$ as in~\eqref{eqn:lagtori}. We refer
the reader to~\eqref{equation:skeleton}.

For a family $C_t$ of curves in $X_t$, we denote the
limiting curve as $t\rightarrow 0$ by $C \in X_0$.
Now since the total space $\mathfrak{X}$  is~$\mathbb{Q}$-Gorenstein,
by the adjunction formula the Maslov index of $C$ equals the
Maslov index of $C_t$.

We are interested in counting disks of
Maslov index two
whose boundary lies on the monotone Lagrangian torus $L=L_1(u)$ of $X$.
By deforming the varieties $X$, we have a family of disks of Maslov index
two in the fibers $X_t$ with boundary $L_t$. Using the condition that
the Maslov index is two, these disks will intersect the anti-canonical
divisor $D$ in the toric variety $X_0$ with total intersection index one.

Now \cref{lem:pointcount} tells us that
the limiting curves $C$
intersects exactly one boundary divisor say $D_i$,
thus it lies in the
complement of $D-D_i$ in $X_0$. This
complement is isomorphic to
$\mathbb{C}\times (\mathbb{C}^{\times})^{\dim X_0-1}$.
Conversely for every chart, we have exactly one curve
by direct computation. Now chasing the torus in~\eqref{equation:skeleton}
we can transport the
disk back to $(X_t,L_t)$. This is possible since~\cref{lem:pointcount}
allows us to ignore the $2$-coskeleton in the toric variety~$X_0$.
This completes the proof in the terminal case.
\end{proof}

%\iffalse
%\begin{theorem}\label{theorem:nnugen1}
%  Let $(X,L)$ be a smooth Fano variety along
%  with a monotone Lagrangian torus obtained as above, then
%  \begin{equation}
%  m_{0}(L)(x)=\sum_{i=1}^m \exp(\langle v_i,x \rangle) T^{\ell_i(u)},
%  \end{equation}where~$v_i$ are the generators of the polar dual~$P^\circ$
%  and~$\ell_i$ are the equations defining
%  the facets of the moment polytope~$P$
%  and valuations of $x$ belong to the moment polytope $P$.
%\end{theorem}
%
%Applying \cref{theorem:nnugen1} to our situation
%we obtain the following corollary.
%% using the same change of variables as in \cite[Theorem~3]{MR2879356}.
%\begin{corollary}
%  \label{corollary:nnugen1}
%  Let $(\graph,c)$ be a trivalent graph with one colored vertex,
%  with no separating edges, then the following are equal:
%  \begin{itemize}
%    \item the graph potential~${W}_{\graph,c}$,
%    \item the disk potential of the Lagrangian torus $L_\graph \subset \odd$;
%  \end{itemize}
%  where $L_\graph = \Phi_\graph^{-1}(0)$ is a fiber
%  of the integrable system
%  $\Phi_\graph\colon\odd^\regular \to \Delta_{\graph,c}$
%  associated with Manon's toric degeneration of $\odd$ to $\manontoric{\graph}$.
%\end{corollary}
%\fi

\paragraph{The general case}We now consider the general case as in the statement of Theorem \ref{theorem:gennnu}. We first find generators for $\pi_2(X_0,L)$.

\paragraph{Generators for $\pi_2(X_0,L)$}
We first show that the curves $[C_i]$ constructed before generate $\pi_2(X_0,L_0)$.
\begin{lemma}
  The classes $[C_i]$ generate $\pi_2(X_0,L_0)$.
\end{lemma}

\begin{proof}
  First we consider the case that $X_0$ is smooth
  (or $\mathbb{Q}$-factorial/simplicial).
  In the case of a smooth (or $\mathbb{Q}$-factorial/simplicial) toric variety
  the curves $C_i$ are dual to divisors $D_i$.
  We can associate to $D_i$ classes in $\HH^2(X_0,\mathbb{Z})$
  and there is a short exact sequence
  \begin{equation}
    0 \to M \to \bigoplus_i \mathbb{Z} D_i \to\HH^2(X_0,\mathbb{Z}) \to 0
  \end{equation}
  Here $M$ is the lattice of monomials dual to the lattice $N$,
  and we also consider the dual short exact sequence
  \begin{equation}\label{equation:a2}
  0 \to \HH_2(X_0,\mathbb{Z}) \to \pi_2( X_0,L_0) \to N \to 0
  \end{equation}
  Equivalently in the basis of curves
  \begin{equation}
    \label{equation:aloha2'}
    0 \to \HH_2(X_0,\mathbb{Z}) \to \bigoplus_i \mathbb{Z} C_i \to N \to 0
  \end{equation}
  Comparing~\eqref{equation:a2} with~\eqref{equation:aloha2'}
  we get $\pi_2(X_0,L_0)=\bigoplus_i\mathbb{Z}C_i$.
  Thus we are done when $X_0$ is either smooth or simplicial.

  %\iffalse
  %Note that if the toric variety $X_0$ is not factorial then right term in $M$-sequence
  %is not $H^2(X_0)$ (the class group of Cartier divisors),
  %but instead it is $H_{2n-2}(X_0)$ (the group of classes of Weil divisors).
  %Also generically there is no intersection pairing between curves and Weil divisors.
  %Nevertheless the curves $C_i$ intersect divisors $D_j$
  %transversally in one point if $i=j$ and are disjoint from $D_j$ otherwise.
  %\fi

  For any $\mathbb{Q}$-factorialization $\widetilde{X}_0$
  (i.e.~a triangulation of a fan of $X_0$),
  the proper preimages of $C_i$ and $D_i$ will form a pair of dual bases
  of $\HH_2(\widetilde{X}_0,L_0)$ and $\HH^2(\widetilde{X}_0,L_0)$
  Also note that the pair $(X_0,L_0)$ is homotopy equivalent
  to the pair $(X_0,(\mathbb{C}^\times)^n)$.
  So for non-$\mathbb{Q}$-factorial $X_0$ we have \emph{more} generators than needed for a base,
  and the map $\bigoplus_i \mathbb{Z} C_i \to \pi_2(X_0,L)$ has a kernel.
\end{proof}

\begin{proof}[Proof of~\cref{proposition:magiclemmawithsalt} in the general case]
  Let us analyze a holomorphic irreducible curve $(C,\partial C) \subset (X_0,L_0)$
  of Maslov index two, for which $C\cdot D = 1$,
  where $D$ is the anti-canonical divisor in~$X_0$.

  Choose any small $\mathbb{Q}$-factorialisation  $r\colon X_0'\to X_0$
  so that $r^* D =  \sum_i D'_i$,
  and lift a curve $(C,\partial C) \subset (X_0,L_0)$
  to $(C',\partial C') \subset (X_0',L)$.
  So we have $\HH_2(r)([C']) = [C] \in \HH_2(X_0,L_0)$.
  On $X_0'$ all intersection indices $C'\cdot D'_i$ are well-defined non-negative rational numbers,
  say $p_i = C'\cdot D'_i$, and $\sum_i p_i = 1$.

  Note that the curve (or, rather, effective relative cycle) $\sum p_i C'_i$
  has exactly the same intersection indices with all $D_i$,
  thus we have an equality
  \begin{equation}
    [C'] = \sum p_i [C'_i] \in \HH_2(X_0',L_0;\mathbb{Q})
  \end{equation}
  to which we can apply $\HH_2(r)$ to obtain an equality
  \begin{equation}
    \HH_2(r)([C']) = \sum p_i \HH_2(r)([C'_i]) \in \HH_2(X_0,L_0;\mathbb{Q}),
  \end{equation}
  i.e.
  \begin{equation}
    [C] = \sum p_i [C_i] \in \HH_2(X_0,L_0;\mathbb{Q}).
  \end{equation}

  This implies that the class of any irreducible curve of Maslov index two
  lies in the convex hull of the classes $[C_i]$.
  By considering the boundary
  (which we also could have applied at the level of $X_0'$)
  we obtain
  \begin{equation}
    [\partial C] = \sum p_i [\partial C_i] \in \HH_1(L_0;\mathbb{Q})
  \end{equation}
  so the class of $[\partial C]$ has to be an integer point inside the fan polytope of $X_0$
  (the convex hull of the primitive rays $\rho_i$).

  Now we also see that if $C$ is an irreducible curve of Maslov index two
  with $[\partial C] = \rho_i$ for some $i$,
  then $[C'] = [C'_i] \in \HH_2(X_0',L_0)$,
  hence $C'\cdot D'_k = 0$ for $k\neq i$,
  thus $C'$ is disjoint from $D'_k$ for $k\neq i$,
  so $C = r(C')$ is disjoint from $D_k = r(D'_k)$ for $k\neq i$,
  i.e.~$C$ lies in an affine chart $\mathbb{C} \times (\mathbb{C}^\times)^{n-1}$
  associated with $D_i$ (so $C$ is the basic disk described above).
  This completes the proof.
\end{proof}

%Now as before Proposition~\cref{proposition:magiclemmawithsalt} allows us to ignore the $2-coskeleton$ in the toric variety and hence chasing the torus $L(u) \subset T_V\subset X_0$ as in ~\eqref{equation:skeleton} we can transport the
%disk back to $(X_t,L_t)$.

%% file: sympl-descend.tex
\section{Mirror partners for moduli of vector bundles}
\label{section:periods}
We can now establish the proof of \cref{theorem:mirror-symmetry-introduction}.
It is given as \cref{corollary:equality-coefficients}
and \cref{proposition:equality-periods}.

\paragraph{Descendant invariants}
Let~$X$ be a Fano variety.
We will denote by~$X_{0,1,\beta}$\index{X@$X_{0,1,\beta}$}
the moduli space of stable maps~$f\colon(C;x)\to X$
where~$C$ is a rational curve with a marked point~$x\in C$
such that~$[f(C)]=\beta\in\HH_2(X,\bZ)$ is a fixed homology class.
It has a virtual fundamental class
in degree~$\int_\beta\mathrm{c}_1(X)-2+\dim X$,
and comes with the evaluation and the forgetful morphisms
\begin{equation}
\begin{aligned}
  \ev\colon X_{0,1,\beta}&\to X = X_{0,1,0}, \\
  \pi\colon X_{0,1,\beta}&\to X_{0,0,\beta}.
\end{aligned}
\end{equation}
Denote by~$\psi$ the first Chern class of the universal cotangent line bundle,
which is also the relative dualizing sheaf~$\omega_\pi$.
We can use it to define the \emph{descendant Gromov--Witten invariants for $\beta\in\mathrm{H}_2(X,\bZ)$}.
These count count rational curves of degree $d$ passing through a fixed generic point
with some kind of tangency condition, and are defined as
\begin{equation}
  \label{equation:descendants}
  \begin{aligned}
    p_\beta&:=
    \langle [\pt] \cdot \psi^{d-2} \rangle_{0,1,\beta}^X
    =\int_{[X_{0,1,\beta}]^{\mathrm{vir}}} \ev^*([\pt]) \cup \psi^{d-2}, \\
    \index{p@$p_d$}
    p_d(X)&:= \sum_{\mu(\beta) = 2d} p_\beta,\qquad d\geq 2, \\
  \end{aligned}
\end{equation}
where $[\pt]$ is the class of the point,
i.e.~a homogeneous class such that $\int_X [\pt] = 1$. We moreover set
\begin{equation}
  p_0(X) := 1,\qquad p_1(X) := 0,
\end{equation}
and
\begin{equation}
  \index{c@$c_d$}
  c_d := d! \cdot p_d,
\end{equation}

\begin{remark}
  Under the assumption that $X$
  has a smooth anti-canonical divisor $D$\footnote{Which by the Bertini theorem holds for the varieties we consider in this paper,
  but fails for some other Fano varieties such as the squares of degree one del Pezzo surfaces.}, Mandel
  \cite[Theorem~1.2]{1903.12014} % Mandel
  shows that $c_d$ is the naive number of maps from $\bP^1$ to $X$
  passing through a fixed generic point on $X$ such that the preimage of the divisor $D$
  is a fixed generic $d$-tuple of points on $\bP^1$.
  In particular the number $c_d$ is a non-negative integer.
\end{remark}

\paragraph{Graph potentials as mirror partners}
The next result, 
demonstrates how the combinatorics of graph potentials
controls the enumerative geometry of~$\odd$.
The equality in this corollary corresponds to
the original condition of enumerative mirror symmetry for Fano manifolds
proposed by Eguchi--Hori--Xiong.

Let us briefly recall the notion of
quantum periods for Fano varieties,
as discussed in Coates--Corti--Galkin--Kasprzyk and Galkin--Iritani \cite{MR3469127,MR4384381}. % CCGGK, GI
% We refer the reader to
% \cite{MR3469127,MR1492534} % CCGGK, Fulton-Pandharipande (1998)
% for more complete historical references to quantum cohomology.
The (unregularized and regularized) \emph{quantum periods}
are the generating power series for descendant Gromov--Witten invariants
$p_d$ and $c_d = d! \cdot p_d$ defined in \eqref{equation:descendants}:
\begin{equation}
  \begin{aligned}
    \label{equation:quantum-period}
    \index{G@$\mathrm{G}_X(t)$}
    \mathrm{G}_X(t)           & := 1 + \sum_{d=2}^{+\infty} p_d(X) t^d, \\
    \index{G@$\widehat{\mathrm{G}}_X(\kappa)$}
    \widehat{\mathrm{G}}_X(\kappa) & := 1 + \sum_{d=2}^{+\infty} c_d(X) \kappa^d =
    \frac{1}{\kappa} \int_0^\infty \mathrm{G}_X(t) \exp(t/\kappa) \dd t
  \end{aligned}
\end{equation}

\begin{corollary}
  \label{corollary:equality-coefficients}
  Let $(\graph,c)$ be a colored trivalent graph
  of genus~$g\geq 2$ with one colored vertex.
  Let~$d\geq 2$.
  We have the equality
  \begin{equation}
    \label{equation:equality-coefficients}
    c_d(\odd)  = [W_{\graph,c}^d]_{z^0}
  \end{equation}
  of the Gromov--Witten invariant $c_d$ and the
   constant term of $d$th power of the Laurent
   polynomial $W_{\graph,c}$.

  In particular we obtain that $W_{\gamma,c}$ is a \emph{mirror partner} for $\odd$,
  in the sense of \cite[\S3]{2112.15339}.
\end{corollary}

\begin{proof}
  By \cref{theorem:graphisdisc} we have that
  the graph potential agrees with the Floer
  potential for any trivalent graph
  with no separating edges.   Hence their classical periods agree.

 By applying \cite[Theorem~1.1]{1801.06921}
  (see also Bondal--Galkin \cite{Bondal-Galkin} and Rau \cite{Rau-phd}) to
  the monotone Lagrangian torus constructed in
  \cref{proposition:monotone-torus}
  we obtain that the classical period of the Floer potential
  agrees with the quantum period of~$\odd$.
  The conclusion about the mirror partner now follows for trivalent
  graph with no separating edges.

  Now since by \cref{theorem:graphtqft} the periods of the graph
  potential associated to any trivalent
  graph of genus $g$ with one colored
  vertex are the same,
  the general case follows.
\end{proof}

\paragraph{Consequences}
The unregularized quantum period $\mathrm{G}_X(t)$ is equal to
the fundamental term of Givental's $J$-series $J_X(t)$ from \cite{MR1408320}. % Givental
This implies that $\mathrm{G}_X(t)$
satisfies the quantum differential equation $L(t,\partial_t)$ with an irregular singularity.
Its Fourier--Laplace transform $\widehat{\mathrm{G}}_X(\kappa)$
satisfies the so-called regularized quantum differential equation
$\widehat{L}(\kappa,\partial_\kappa)$
that is expected to be regular.

Under the assumption that all eigenvalues
of the quantum multiplication operator $\star_0 \mathrm{c}_1(X)$ are distinct
the respective regularized connection is Fuchsian.
However for $X = \odd$ only two eigenvalues have multiplicity one,
as can be computed by result due to Mu\~noz in \cite{MR1695800}
(see \cite[\S3.1]{gp-decomp} for additional discussion).
But under the identification of the period sequences we can show that
the principle irreducible component of
the regularized quantum connection is
a summand of a Picard--Fuchs equation
and, as a corollary, indeed has regular singularities.

Summing \eqref{equation:equality-coefficients} with weights $\frac{t^d}{d!}$ or $\kappa^d$
we get an equivalent reformulation of \cref{corollary:equality-coefficients}
in terms of the equality of the respective period series.
\begin{proposition}
\label{proposition:equality-periods}
Let $(\graph,c)$ be a colored trivalent graph
of genus~$g\geq 2$ with one colored vertex.
%Assuming\fixthis{no longer needed} \cref{conjecture:small-resolution} we have the equality
\begin{equation}
\begin{aligned}
  \mathrm{G}_{\odd}(t) &= \int_{|z|=1} \exp(t\cdot W_{\graph,c}(z))\frac{\mathrm{d}z}{z} , \\
  \widehat{\mathrm{G}}_{\odd}(\kappa)
  &=
  \pi_{\widetilde{W}_{\graph,c}}(\kappa) \,\qquad \text{for }|\kappa| \leq \frac{1}{8g-8},
\end{aligned}
\end{equation}
of the unregularized (resp. regularized) quantum period of~$\odd$
and the oscillating integral (resp. period)
of the graph potential~$\widetilde{W}_{\graph,c}$.
\end{proposition}

\begin{corollary}
  \label{corollary:ode}
  The regularized quantum period
  is a solution of an ordinary differential equation
  with regular singularities.
\end{corollary}

\begin{proof}
  By \cref{proposition:equality-periods} we have the identification of the regularized quantum period with the period of the graph potential. As discussed in \cite[\S3]{MR3469127}, the period of a Laurent polynomial~$f$ satisfies an ordinary differential equation, given by the Picard--Fuchs operator. The local system of solutions is a summand of a polarised variation of Hodge structures, which by \cite[Theorem~4.5]{MR0498581} has regular singularities.
\end{proof}

\begin{remark}
  \label{remark:asymptotics}
  This also allows us to use methods of convex geometry and random walks
  to estimate the asymptotics of the coefficients,
  see e.g. \cite[Lemma 3.13]{MR4384381}. % GI
  In particular,
  the limit $\lim_{k\to\infty} ([W^{4k}]_{z^0})^{1/4k}$
  exists and is equal to $8(g-1) = T_{\odd}$.
\end{remark}

In \cref{table:quantum-periods-odd,table:quantum-periods-even}
we give the period sequences for the even and odd graph potentials.
In the odd case this in turn coincides with
the quantum period sequence for~$\odd$ by \cref{proposition:equality-periods}.
These numbers were computed using the method introduced in \cite{gp-tqft}
for computing periods of graph potentials using the graph potential TQFT.
This efficient method for computing periods from \cite{gp-tqft}
could be useful for algorithmically determining
the quantum differential equation for the quantum periods of~$\odd$.

In \cref{section:explicit-calculations}
we will give some explicit calculations for the quantum period sequence,
highlighting some of its patterns which one can read off from the tables.

\begin{remark}
  \label{remark:speculation-on-even}
  We have mostly focused on~$\odd$,
  and only briefly commented on some aspects of~$\even$ in \cref{remark:even-symplectic}.
  There are further alternative points of view on~$\even$,
  namely as a Poisson variety (by the work of Huebschmann, see e.g.~\cite{MR1324125})
  or as a stack,
  which could take care of the singularities
  present in~$\even$ for~$g\geq 3$,
  and the discrepancy between
  the period sequence of the uncolored graph potential for~$g=2$
  and the quantum period of~$\even\cong\bP^3$.
  Indeed, by \cite[Tables~1 and~2]{gp-tqft}
  which are recalled in \cref{table:quantum-periods-even}
  we have that the non-zero terms~$p_{4k}$ in
  the classical period sequence of the uncolored graph potential for~$g=2$
  are
  \begin{equation}
    1,384,645120,1513881600,4132896768000,\ldots
  \end{equation}
  whereas the non-zero terms~$p_{4k}$ in
  the regularised quantum period sequence for~$\mathbb{P}^3$ (see e.g.~\cite[\S1]{MR3470714}) are
  \begin{equation}
    1,24,2520,369600,63063000,\ldots
  \end{equation}
  We can make two comments:
  \begin{itemize}
    \item To fix the discrepancy for~$g=2$,
      it suffices to rescale the graph potential
      described in \cite[Example~2.6]{gp-tqft}
      by~$1/2$ to get the correct period sequence.

    \item For $g\geq 3$ and a connected graph $\graph$ without separating edges
      the Laurent polynomial $W_\graph$ is supported in the vertices of its Newton polytope.
      There are $8g-8$ non-vanishing coefficients, each equal to one,
      which are in bijection with facets of $\Delta_\graph$
      and with irreducible torus-invariant Weil divisors on $X_0$.
      A fiber of a moment map bounds $8g-8$ holomorphic disks,
      each transversally intersecting exactly one of toric divisors,
      and if $X_0$ has a small resolution of singularities
      this shows that there are no other holomorphic disks
      of Maslov index $2$.
  \end{itemize}
  We observe that the graph potential~$W_{\graph,0}$
  reflects some aspects of the symplectic geometry of the singular variety~$\even$.
  Since Fukaya categories of singular varieties are not so well established,
  it seems reasonable to start with categories of matrix factorizations
  of the graph potentials $W_{\graph,0}$
  \emph{instead of} the not so well-defined categories $\Fuk{\even}$,
  in order to gain further insight into their (conjecturally common) properties.
  This gives further evidence for the speculation on the Atiyah--Floer conjecture
  and Donaldson--Floer theories in \cite[\S1.3]{gp-tqft}.
\end{remark}

%% file: sympl-singular.tex
\section{Analyzing the singularities}
\label{section:singularities}
Let~$(\graph',c)$ be the colored trivalent graph constructed from a trivalent graph~$\graph$ with at most one half-edge in \cref{subsection:manon}. We wish to study the singularities of the toric variety~$\manontoric{\graph}\cong\gptoric{\graph'}{c}{M}$.

Before we discuss the general picture, let us discuss what happens for the toric degenerations of~$\odd$ and~$\even$ for~$g=2$. There are two graphs, and as will be clear from \cref{theorem:terminal-criterion} and \cref{remark:small-resolution-check}, the distinction between the behavior of the toric degenerations for~$\odd$ depending on the choice of graph is representative of what happens in general.

\begin{figure}[t!]
  \centering

  \begin{subfigure}[b]{\textwidth}
    \centering
    \begin{tikzpicture}[scale=1.75]
      \node[vertex] (A) at (0,0) {};
      \node[half-vertex] at (A) {};
      \node[vertex] (B) at (1,0) {};
      \node[half-vertex] at (B) {};

      \draw (A) +(0,-0.1) node [below] {$1$};
      \draw (B) +(0,-0.1) node [below] {$2$};

      \draw (A) edge [bend left]  node [above]      {$a$} (B);
      \draw (A) edge              node [fill=white] {$b$} (B);
      \draw (A) edge [bend right] node [below]      {$c$} (B);
    \end{tikzpicture}
    \caption{Theta graph}
    \label{figure:theta-graph}
  \end{subfigure}

  \begin{subfigure}[b]{\textwidth}
    \centering
    \begin{tikzpicture}[scale=1.75]
      \node[vertex] (A) at (0,0) {};
      \node[half-vertex] at (A) {};
      \node[vertex] (B) at (1,0) {};
      \node[half-vertex] at (B) {};

      \draw (A) +(0,-0.1) node [below] {$1$};
      \draw (B) +(0,-0.1) node [below] {$2$};

      \draw (A) edge [loop left]  node [left]       {$b$} (A);
      \draw (A) edge              node [fill=white] {$a$} (B);
      \draw (B) edge [loop right] node [right]      {$c$} (B);
    \end{tikzpicture}
    \caption{Dumbbell graph}
    \label{figure:dumbbell-graph}
  \end{subfigure}

  \caption{Trivalent genus 2 graphs}
  \label{figure:trivalent-g-2}
\end{figure}
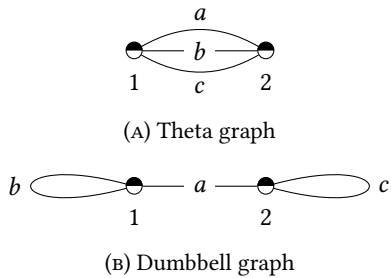
\subsection{Terminal singularities}\label{sec:terminal}
Recall from \cite[Proposition~11.4.12(ii)]{MR2810322} that for a~$\bQ$\dash Gorenstein toric variety, having terminal singularities is equivalent to the only non-zero lattice points in the fan polytope being its vertices. We can use this criterion, as by \cite[Theorem~12]{MR4025429} the homogeneous coordinate rings of the toric degenerations~$\manontoric{\graph}$ are all arithmetically Gorenstein and in particular the toric degenerations themselves are Gorenstein.

\begin{example}
  \label{example:g-2-singularities}
  If~$\graph$ is the Theta graph from \cref{figure:theta-graph} and there is one colored vertex, then the toric variety we obtain is the singular intersection of quadrics in~$\bP^5$ given by~$x_0x_5-x_1x_2=0$ and~$x_3x_4-x_1x_2=0$ \cite{MR3339536}.

  The  polar dual of the moment polytope is a cube with vertices~$(\pm 1, \pm 1, \pm 1).$ This variety has~6~ordinary double points and hence has a small resolution. In particular, the variety also has terminal singularities. This case is studied by Galkin \cite{Galkin-phd} and also by Nishinou--Nohara--Ueda \cite{MR3339536, MR2879356}. The results in this section and the next can be seen as a generalization from~$g=2$ to~$g\geq 3$.

  On the other hand, if~$\graph$ is the dumbbell graph from \cref{figure:dumbbell-graph} with one colored vertex, the toric variety is given by the equations~$x^2 = zw$ and $zw = uv$ in $\bP^5$ \cite{MR3339536}. The dual of the moment polytope of the toric variety is the convex hull of the points
  \begin{equation}
    (1,2,0),(1,-2,0),(-1,0,0),(-1,0,-2),(1,0,0),(-1,0,2).
  \end{equation}
  The polytope contains both the lattice points~$(-1,0,-2)$ and~$(-1,0,2)$, but then it also contains~$(-1,0,0)$, and hence it cannot be terminal. The distinction between the toric varieties obtained from these two graphs is explained by \cref{theorem:terminal-criterion}.
\end{example}
The case of~$\even$ for~$g=2$ is a bit pathological, as we by \cite[Theorem 1]{MR0242185} have an isomorphism~$\even\cong\bP^3$. Moreover the locus of semistable bundles up to~$\mathrm{S}$-equivalence is identified with the Kummer surface in~$\bP^3$.
\begin{example}
  If~$\graph$ is the Theta graph and there are no colored vertices, then the toric variety is just~$\bP^3$. Clearly it is smooth and hence terminal.

  If~$\graph$ is the dumbbell graph and there are no colored vertices, then the dual of the moment polytope is the convex hull of the points
  \begin{equation}
    (1,2,0),(1,-2,0),(-1,0,0),(1,0,2),(1,0,0),(1,0,-2).
  \end{equation}
  The polytope contains the lattice points~$(1,0,0)$ which is exactly the half sum of the lattice points~$(1,0,2)$ and~$(1,0,-2)$. Hence the toric variety is non-terminal.
\end{example}

Then the dichotomy observed for~$g=2$ from \cref{example:g-2-singularities} in the case of 1 colored vertex is illustrative for what happens for~$g\geq 3$.
\begin{theorem}
  \label{theorem:terminal-criterion}
  Let~$(\graph,c)$ be a colored  trivalent graph of genus~$g\geq 3$. The toric variety~$X_{P_{\graph,c}, M_{\graph}}$ has terminal singularities if and only if~$\graph$ has no separating edges.
\end{theorem}

Before we prove this statement we give some preliminary lemmas. The following two lemmas work for an arbitrary integer~$n$, even though we are mostly interested in the case~$n=3g-3$.

\begin{lemma}
\label{lemma:outside-convex-hull}
Let~$n\geq 3$ and consider~$\bZ^n$.
Consider the set~$S$ of~$2^3\binom{n}{3}$ points of the form
$\pm e_i\pm e_j\pm e_k$ for~$1\leq i,j,k\leq n$ distinct.
Let~$p=s_ie_i+s_je_j+s_ke_k$ be an element of~$S$,
then~$p$ does not lie in the convex hull of~$S\setminus\{p\}$.
\end{lemma}

\begin{proof}
  Consider the function~$h=s_i e_i^\vee+s_j e_j^\vee+s_k e_k^\vee$. We have that~$h(s_ie_i+s_je_j+s_ke_k)=3$, and~$h(q)<3$ for~$q\in S\setminus\{p\}$.
\end{proof}

We will denote~$\Pi_n$ the polytope given by the convex hull of the set~$S$ from \cref{lemma:outside-convex-hull}.

\begin{lemma}
  \label{lemma:studying-Pi-n}
  With respect to the standard lattice~$\bZ^n$ the polytope~$\Pi_n$ contains precisely the following lattice points:
  \begin{itemize}
    \item $0\in\Pi_n$
    \item $\pm e_i\in\Pi_n$ for~$1\leq i\leq n$
    \item $\pm e_i\pm e_j\in\Pi_n$ for~$1\leq i,j\leq n$ distinct
    \item $\pm e_i\pm e_j\pm e_k\in\Pi_n$ for~$1\leq i,j,k\leq n$ distinct.
  \end{itemize}
  Moreover,
  \begin{itemize}
    \item for all $1\leq i\leq n$ the lattice points $\pm e_i$ lie in the convex hull of~$\pm e_i\pm e_j\pm e_k$ for~$1\leq j,k\leq n$ distinct and~$i\neq j,k$;
    \item for all $1\leq i,j\leq n$ distinct the lattice points $\pm e_i\pm e_j$ lie in the convex hull of~$\pm e_i\pm e_j\pm e_k$ for~$1\leq k\leq n$ such that~$k\neq i,j$.
  \end{itemize}
\end{lemma}

\begin{proof}
  The function~$h=\sum_{i=1}^n(e_i^\vee)^2$ is convex, and~$h(\pm e_i\pm e_j\pm e_k)=3$, hence all these lattice points lie on a sphere of radius~1 and they cannot be contained in each other convex hulls. This function takes on the values~$0,1,2,3$ on the lattice points listed in the statement of the lemma.

  To see that there are no other, observe that for two vectors~$u,u'$ whose coefficients in the standard basis belong to~$\{-1,0,1\}$, we have that~$(u',u)\leq(u,u)$ in the standard bilinear pairing. If we write~$u=\sum_{i=1}^{2^3\binom{n}{3}}\alpha_iv_i$ where~$v_i$ are the vertices of~$\Pi_n$, such that~$\alpha_i\geq 0$ for all~$i$ and~$\sum_{i=1}^{2^3\binom{n}{3}}\alpha_i=1$, we can apply this observation to~$u'=e_i$ for all~$i$ and we get
  \begin{equation}
    (u,u)=\sum_{i=1}^{2^3\binom{n}{3}}\alpha_i(v_i,u)\leq\sum_{i=1}^{2^3\binom{n}{3}}\alpha_i(u,u)=(u,u)
  \end{equation}
  hence either~$\alpha_i=0$ or~$(v_i,u)=(u,u)$.

  The final claim is immediate.
\end{proof}

We need one more lemma, giving a homological interpretation to the notion of a separating edge, using the notation of \cref{subsection:graph-potentials}.
\begin{lemma}
\label{lemma:separating-characterization}
Let~$\graph$ be a trivalent graph.
Let~$e\in E$ be an edge,
and consider the associated indicator function~$e^\vee\in\CC^1(\graph,\bZ)=\widetilde{N}_\graph$. The vectors $e^{\vee}$ and $-e^{\vee}$ are always inside the polytope. Moreover~$e^\vee\in N_\graph\subseteq\widetilde{N}_\graph$ if and only if~$e$ is separating.
\end{lemma}

\begin{proof}
  Assume that~$e$ is separating, i.e.~by removing it we get two connected components which we will denote~$\graph'$ and~$\graph''$. Choose an orientation of the edges of~$\graph$, and assign~$1$ to an incoming half-edge, and~$-1$ to an outgoing half-edge. For~$v\in\graph'$ we consider the sublattice~$N_v$, and in there consider the sum
  \begin{equation}
    \label{equation:lattice-element}
    \sum_{v\in \graph'}p_{v,s(v,o)} \in N_{\graph}
  \end{equation}
  where~$o$ denotes the orientation scheme we choose on the graph~$\graph$,~$s(v,0) \in\mathrm{C}_1(\graph_v,\bF_2)$ defined by~$o$ and~$p_{v,s(v,0)}$ be as in~\eqref{equation:vertices}. By construction the sum is, up to a sign, equal to~$e^{\vee}$, hence~$e^\vee\in N_\graph$.

  Assume that~$e^\vee\in N_\graph$, and for the sake of contradiction assume that~$e$ is non-separating. Then there exists a cycle~$Z$ in~$\graph$ (without repetition of vertices), defining an element of~$\CC_1(\graph,\bZ)$ by assigning~1~to all edges of the graph.

  On the other hand, there exists a chain of inclusions
  \begin{equation}
    2\widetilde{N}_\graph\subseteq N_\graph\subseteq\widetilde{N}_\graph
  \end{equation}
  induced by the identity~$\pm2x_i=(\pm x_i\pm x_j\pm x_k)+(\pm x_i+\mp x_j\mp x_k)$, where the~$x_i$'s are as in \cref{subsection:graph-potentials}. These induce an isomorphism
  \begin{equation}
    \widetilde{N}_\graph/2\widetilde{N}_\graph\cong\CC^1(\graph,\bZ/2\bZ)
  \end{equation}
  which then induces an isomorphism
  \begin{equation}
    N_\graph/2\widetilde{N}_\graph\cong\dd(\CC^0(\graph,\bZ/2\bZ)) \\
  \end{equation}
  For~$e^\vee$ to be in~$N_\graph$ (and not just~$\widetilde{N}_\graph$) it must evaluate to zero on cycles since
  \begin{equation}
    N_\graph/2\widetilde{N}_\graph \cong\dd(\CC^0(\graph,\bZ/2\bZ)).
  \end{equation}
  But by construction~$e^\vee(Z)$ is odd. This contradiction shows that~$e^\vee$ cannot be an element of~$N_\graph$.
\end{proof}

\begin{proof}[Proof of \cref{theorem:terminal-criterion}]We give the following general Lemma describing the lattice points of the Newton polytope of a graph potential. The proof of \cref{theorem:terminal-criterion} is an immediate corollary.
\end{proof}
\begin{lemma}\label{lemma:grandcentral}
Let $\gamma,c$ be a  connected trivalent colored graph of genus $g>2$ and consider the Newton polytope~$P^\circ_{\gamma,c}$of the graph potential. The non-zero integer points are the following:
\begin{itemize}
	\item rays (vertices of the polytope)
	\begin{itemize}
		\item with every non-loop vertex of the graph~$\gamma$ there are 4 associated rays.
		\item with every loop vertex of graph~$\gamma$ there are 2 associated rays.
		\end{itemize}
	\item non-rays lattice points of the form $e_b$ and $-e_b$ associated to every separating edge $b$ in the graph.
\end{itemize}
In particular total number of rays equals $8g-8-2\#\text{loops}$ and the number of lattice points other than zero and rays is $2\#\text{bridges}$.
	\end{lemma}
	\begin{proof}
  Assume that~$\graph$ has no separating edges. Then in particular it does not contain loops, and hence the set of vertices of ~$P_{\graph,c}^\circ$ is a subset of~$\{\pm e_i\pm e_j\pm e_k\}$, where~$e_i$ are standard vectors in~$\widetilde{N}_{\graph}\cong\bZ^{\#E}$. By \cref{lemma:outside-convex-hull} we have that all the points~$p(v,s)$ given by~\eqref{equation:vertices} are vertices. It remains to show that there are no other lattice points on the boundary of this polytope.

  For this, by \cref{lemma:studying-Pi-n} it suffices to show that no points of the form~$\pm e_i$ and~$\pm e_i\pm e_j$ is a lattice point in the convex hull of the vertices of~$P_{\graph,c}^{\circ}$. Since the graph $\graph$ has no separating edges, it follows that from \cref{lemma:separating-characterization} that~$\pm e_i$ cannot be a lattice point in $N_{\graph}$.

  Now say for example that~$e_i+e_j$ would be a lattice point of the polytope, and let~$v$ be a vertex adjacent to~$e_i$ and~$e_j$. Then we can write
  \begin{equation}
    e_i+e_j=\frac{1}{2}\left( (e_i+e_j+e_k)+(e_i+e_j-e_k) \right)
  \end{equation}
  where~$e_k$ is the third edge adjacent to~$v$. However, the only way in which we can do this if~$e_i$, $e_j$ and~$e_k$ have the same end points, but then we obtain a subgraph of genus~$2$ in~$\graph$ which is a contradiction.

  The converse follows from \cref{lemma:separating-characterization} and~\eqref{equation:lattice-element}.
\end{proof}

As an application of \cref{theorem:terminal-criterion}, we can give an alternative proof of the following result due to Kiem--Li \cite[Corollary 5.4]{MR2099191} (albeit only for a generic curve, rather than for every curve). Their proof is based on understanding the discrepancy between the canonical bundle of~$\even$ and its Kirwan desingularization. In our setup it rather follows from Kawamata's inversion of adjunction for terminal singularities.
\begin{corollary}[Kiem--Li]\label{cor:KiemLi}
Let~$C$ be a generic smooth projective curve of genus~$g\geq 2$.
Then~$\even$ has at most terminal singularities.
\end{corollary}

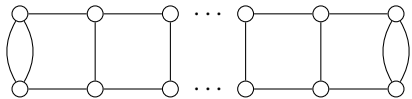
\begin{figure}[t!]
  \centering
  \begin{tikzpicture}[scale=1]
    \node[vertex] (A) at (0,0) {};
    \node[vertex] (B) at (1,0) {};
    \node[vertex] (C) at (2,0) {};
    \node[vertex] (D) at (3,0) {};
    \node[vertex] (E) at (4,0) {};
    \node[vertex] (F) at (5,0) {};
    \node[vertex] (G) at (0,1) {};
    \node[vertex] (H) at (1,1) {};
    \node[vertex] (I) at (2,1) {};
    \node[vertex] (J) at (3,1) {};
    \node[vertex] (K) at (4,1) {};
    \node[vertex] (L) at (5,1) {};

    \node at (2.5,0) {$\ldots$};
    \node at (2.5,1) {$\ldots$};

    \draw (A) edge (B);
    \draw (B) edge (C);
    \draw (D) edge (E);
    \draw (E) edge (F);

    \draw (A) edge [bend left]  (G);
    \draw (A) edge [bend right] (G);
    \draw (B) edge (H);
    \draw (C) edge (I);
    \draw (D) edge (J);
    \draw (E) edge (K);
    \draw (F) edge [bend left]  (L);
    \draw (F) edge [bend right] (L);

    \draw (G) edge (H);
    \draw (H) edge (I);
    \draw (J) edge (K);
    \draw (K) edge (L);
  \end{tikzpicture}
  \caption{Ladder graph}
  \label{figure:ladder}
\end{figure}

\begin{proof}
  By Kawamata \cite[Theorem~1.5]{MR1718145} we know that having at most terminal singularities in the central fiber implies that a general fiber has at most terminal singularities.

  Manon's construction provides us with a family of varieties whose special fiber is toric and whose general fibers are of the form~$\even$, where~$C$ runs over smooth projective curves of genus~$g$ in the smooth locus of the family.

  By \cref{theorem:terminal-criterion} it suffices to choose a trivalent graph~$\graph$ of genus~$g$ without separating edges, so that the toric special fiber has at most terminal singularities. For this we can consider for instance the ladder graph as in \cref{figure:ladder}. Then in the family given by the toric degeneration the central fiber has at most terminal singularities, hence so does~$\even$ in a Zariski open neighbourhood of the special fiber.
\end{proof}

\subsection{Small resolutions of singularities}
In the remainder of this section,
we further discuss when the toric degenerations constructed by Manon admit a small resolution of singularities.
We will not use results of this section in the rest of the article.
%This is  related to the general question of comparing descendant Gromov-Witten invariants of a Fano variety~$F$ and periods of the Newton polynomials of the polytope associated to a  toric degeneration.
Toric degenerations admitting a small resolution of singularities
are useful in the study of mirror symmetry, as e.g.~showcased in \cite{MR1756568,MR2112571,MR2609019}.
We refer the reader to~\cref{section:potentials} for a more precise result
on how existence of small resolution implies a direct way of computing descendant Gromov-Witten invariants.
%The following is motivated by a crucial observation of Bondal--Galkin \cite{Bondal-Galkin}. They observed that the work of Nishinou--Nohara--Ueda \cite{MR2609019} can be interpreted as a candidate for this mirror Laurent polynomial. In \cite{MR2609019}, the authors construct a Laurent polynomial  for any smooth Fano variety $F$ by counting holomorphic discs with a fixed Lagrangian boundary in the Fano variety $F$. They construct their  Laurent polynomial by first constructing a toric degeneration $X$ of the Fano variety $F$ and consider the associated toric complete integrable system (see \cref{definition:integrable-system}).

%Recall that a resolution of singularities for a variety~$X$
%is said to be \emph{small} if there is no exceptional divisor.
%In particular the resolution is crepant. % Small resolutions play fundamental roles in the minimal model program, McKay correspondences and orbifold Gromov--Witten theory.
%The existence of a small resolution
%is a hard question in general
%and it is a delicate condition
%on the singularities of the variety $X$.
% Even for finite quotient singularities of threefolds,
% existence of small resolutions
% has garnered a lot of attention
% \cite{MR3771150, MR1824990, MR1380512}.
% It is well-known that if $X$ has
% a $\bQ$-factorial terminal singularity,
% then $X$ does not have a small resolution.
% On the other hand if $X$ has Gorenstein terminal singularity,
% then it may admit a small resolution.
%We restrict ourselves to a special class
%of Gorenstein terminal toric varieties.
Recall that (torically) resolving the singularities of a toric variety $X$
amounts to subdividing each cone in the fan of the toric variety
into a union of cones generated by a basis of the lattice.
A toric resolution is small if we can find
such a subdivision without adding new rays.
This is in general a non-trivial combinatorial problem.

 We consider the following class of trivalent graphs.

\begin{definition}
  Let~$\graph$ be a trivalent graph.
  We say it is \emph{maximally edge-connected}
  if one needs to remove at least 3 edges to make it disconnected.
\end{definition}

Maximally connected trivalent graphs~$\graph$ have a very interesting positivity property.
Let~$C_0$ be a nodal curve whose dual graph is a maximally edge-connected trivalent graph,
then by \cite[Proposition 2.5]{MR1097026} the canonical bundle of the curve~$C_0$ is very ample.
We do not precisely know which role maximally edge-connectedness plays
in the combinatorial description of the cones of toric varieties~$\gptoric{\graph}{c}{M}$,
but motivated by the genus two case and experimental evidence,
we conjecture the following.

\begin{conjecture}
  \label{conjecture:small-resolution}
  Let~$(\graph,c)$ be a maximally edge-connected trivalent graph.
  Then the graph potential toric variety~$\gptoric{\graph}{c}{M}$
  admits a small resolution of singularities.
  Combinatorially,
  we conjecture that there exists a refinement of the fan of~$\gptoric{\graph}{c}{M}$
  into a simplicial fan such that the resulting toric variety is smooth.
\end{conjecture}

\begin{figure}[t!]
  \centering
  \begin{subfigure}[b]{\textwidth}
    \centering
    \begin{tikzpicture}[scale=1]
      \node[vertex] (A) at (0,0) {};
      \node[vertex] (B) at (2,0) {};
      \node[vertex] (C) at (1,0.5) {};
      \node[vertex] (D) at (1,1.5) {};

      \draw (A) edge (B);
      \draw (A) edge (C);
      \draw (A) edge (D);
      \draw (B) edge (C);
      \draw (B) edge (D);
      \draw (C) edge (D);
    \end{tikzpicture}
    \caption{Tetrahedron graph}
    \label{subfigure:tetrahedron}
  \end{subfigure}

  {\ }

  \begin{subfigure}[b]{\textwidth}
    \centering
    \begin{tikzpicture}[scale=1]
      \node[vertex] (A) at (0,0) {};
      \node[vertex] (B) at (1,0) {};
      \node[vertex] (C) at (2,0) {};
      \node[vertex] (D) at (0,1) {};
      \node[vertex] (E) at (1,1) {};
      \node[vertex] (F) at (2,1) {};

      \draw (A) edge (D);
      \draw (A) edge (E);
      \draw (A) edge (F);
      \draw (B) edge (D);
      \draw (B) edge (E);
      \draw (B) edge (F);
      \draw (C) edge (D);
      \draw (C) edge (E);
      \draw (C) edge (F);
    \end{tikzpicture}
    \caption{Bipartite graph of $g=4$}
    \label{subfigure:bipartite-g-4}
  \end{subfigure}
  \caption{Graphs giving small toric resolutions}
  \label{figure:small-graphs}
\end{figure}
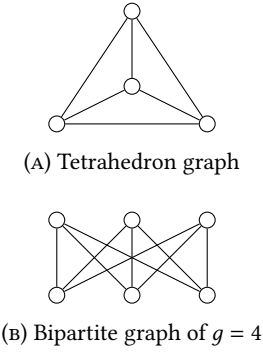

\begin{remark}
  \label{remark:small-resolution-check}
  For~$g=2$ this follows from the discussion in \cref{example:g-2-singularities} for the Theta graph, which is the unique maximally connected trivalent graph of genus~2.

  Using computer algebra we have checked \cref{conjecture:small-resolution} also in genus~3 for the tetrahedron graph (depicted in \cref{subfigure:tetrahedron}), and in genus~4 for the bipartite graph (see \cref{subfigure:bipartite-g-4}) with an odd number of vertices colored, by explicitly constructing a simplicial subdivision of the fan polytope. Moreover, in the genus three, four and five cases, we can find non-maximally connected graphs with no separating edges such that the corresponding toric varieties do not admit a small resolution of singularities.
\end{remark}

%% file: sympl-lines.tex
% !TEX encoding = UTF-8 Unicode

\section{Some explicit calculations for periods}
\label{section:explicit-calculations}
In this section we collect some observations about
the (non-)vanishing
and values of the (quantum) periods studied in the main body of the paper.
These explain some of the patterns which one can observe in \cref{table:quantum-periods-odd,table:quantum-periods-even}.

First we will consider the second and fourth (quantum) period,
in \cref{subsection:general-patterns} we discuss some patterns which hold more generally.

\subsection{Second quantum period}
\label{subsection:second-period}
The first interesting quantum period of~$\odd$ is~$p_2$.
We will show in \cref{corollary:p2-vanishing} using algebro-geometric methods that it is always zero, as soon as~$g\geq 3$.
This is suggested by the combinatorial calculation in \cref{proposition:p2} for the classical period~$\pi_2$.

In \cref{subsection:general-patterns} we will extend this combinatorial vanishing result for the period~$\pi_2$
to the vanishing of~$\pi_{4n+2}$ for all~$n\geq 0$ and~$g\gg 0$.

\begin{proposition}
  \label{proposition:p2}
  Let~$\graph$ be a trivalent graph of genus~$g$.
  The second period of~$\widetilde{W}_{\graph,1}$ associated to~$\odd$ is given by
  \begin{equation}
    \pi_2(\widetilde{W}_{\graph,1})=
    \begin{cases}
      8 & g=2 \\
      0 & g\geq 3.
    \end{cases}
  \end{equation}
\end{proposition}

\begin{proof}
  For~$g=2$ this is immediate from taking the square of the graph potential for the colored Theta graph in \cite[Example~2.6]{gp-tqft}
  and pairing inverse monomials together.

  Similarly, for~$g\geq 3$ we are pairing inverse monomials together
  by squaring and taking the constant term.
  But because we assume that the graph~$\graph$ is connected,
  the monomials contributed by the unique colored vertex have no corresponding inverse monomial in the graph potential.
\end{proof}

The second quantum period on the other hand counts
the number of lines (in the polarization given by the ample generator~$\Theta$ of~$\Pic\odd$)
through a generic point of the variety~$\odd$.
We wish to explain the vanishing for~$g\geq 3$ geometrically,
by using the explicit geometry in the hyperelliptic case.

\begin{remark}
  \label{remark:g-2-agreement}
  For~$g=2$ we have that~$\odd$ is a Fano~3\dash fold,
  and its second (regularized) quantum period
  can be read off in \cite[\S6]{MR3470714}.
  More generally, in \cite[Corollary~4.10]{gp-tqft}
  we have obtained an expression of
  (the inverse Laplace transform of)
  the period sequence
  which agrees with that of
  the quantum period sequence
  from \S6~of op.~cit.
\end{remark}
%\fixthis{And QH was computed by Beauville in 90, J-series by Givental in 95+.
%Galkin in STD constructed small toric degeneration and associated Laurent polynomial,
%and proved that its period equals to the quantum period using quantum Lefschetz and its inversion
%(index equals two and both smooth threefold and its toric degeneration have an isomorphic
%hyperplane section, del Pezzo of degree four --- I needed this argument for del Pezzo threefold of degree five,
%but the argument is written uniformly for all del Pezzo threefolds).}

\paragraph{Geometric setup}
In \cite{MR0429897} the relationship between
a hyperelliptic curve~$C$,
the geometry of the intersection of 2 quadrics
(given by quadratic forms~$q_1$ and~$q_2$, determined by the hyperelliptic curve~$C$)
in~$\bP^{2g+1}$,
and the moduli space~$\odd$ is described.
We will use this to sketch a geometric argument for the vanishing of the second quantum period for~$g\geq 3$.

By Theorem~1 of op.~cit.~$\odd$ is isomorphic to
the closed subvariety of~$\Gr(g-1,2g+2)$ parametrizing~$\bP^{g-2}\subseteq\bP^{2g+1}$ contained in the quadrics.
Likewise, by Theorem~2 of op.~cit.~the Jacobian of~$C$
(or rather, the torsor~$\Pic^gC$)
is isomorphic to a closed subvariety of~$\Gr(g,2g+2)$ parametrizing~$\bP^{g-1}\subseteq\bP^{2g+1}$ contained in the quadrics.
These two closed subvarieties are given as the intersection~$\OGr(k,2g+2;q_1)\cap\OGr(k,2g+2;q_2)$,
where~$\OGr(k,2g+2;q_i)$ is the orthogonal Grassmannian embedded in~$\Gr(k,2g+2)$, for~$k=g-1,g$.

Now consider the incidence correspondence
\begin{equation}
  \begin{tikzcd}
    \Fl(g-1,g,2g+2) \arrow[d] \arrow[r] & \Gr(g-1,2g+2) \\
    \Gr(g,2g+2).
  \end{tikzcd}
\end{equation}
The vertical map is a~$\bP^{g-1}$\dash bundle,
the horizontal map is a~$\bP^{g+3}$\dash bundle.

The restriction of the vertical map to~$\Pic^g(C)$
corresponds to the interpretation of the incidence correspondence from \cite[Remark~5.10(I)]{MR0429897}.
For this interpretation we will consider~$\odd$ for a line bundle~$\cL$ of degree~$2g+1$,
and consider non-split extensions
\begin{equation}
  0\to j\to\cV_j\to j^\vee\otimes\cL\to 0
\end{equation}
for~$j\in\Pic^gC$, which are always stable.

The~$g-1$\dash dimensional projective space~$\bP_j:=\bP(\HH^1(C,j^2\otimes\cL)^\vee)$
parametrizing these extensions gives rise to a~$\bP^{g-1}$\dash bundle~$\cP$.
If we let~$\cM$ be the Poincar\'e bundle on~$C\times\Pic^g(C)$,
then~$\cP\cong\bP_C(\cE)$ where~$\cE:=\RR^1p_{2,*}(\cM^{\otimes2}(-1))$.
Then \cite[Remark~5.10(I)]{MR0429897} gives the following
\begin{lemma}
  \label{lemma:P-interpretation}
  We have that~$\cP$ fits into the fiber product diagram
  \begin{equation}
    \begin{tikzcd}
      \cP \arrow[r, hook] \arrow[d] & \Fl(g-1,g,2g+2) \arrow[d] \\
      \Pic^g(C) \arrow[r, hook] & \Gr(g,2g+2).
    \end{tikzcd}
  \end{equation}
\end{lemma}
The horizontal map in the incidence correspondence gives a morphism~$f\colon\cP\to\odd$.

We can consider other incidence correspondences.
Before proving the interpretation in \cref{proposition:012-interpretation},
we will give an easier description for the Fano variety of lines on~$\odd$.
This is a natural continuation of the discussion in \cite[Remark~5.10]{MR0429897}.

\begin{proposition}
  \label{proposition:fano-variety-interpretation}
  The Fano variety of lines~$\fano(\odd)$ fits in the fiber product diagram
  \begin{equation}
    \begin{tikzcd}
      \fano(\odd) \arrow[r, hook] \arrow[d] & \Fl(g-2,g,2g+2) \arrow[d] \\
      \Pic^g(C) \arrow[r, hook] & \Gr(g,2g+2).
    \end{tikzcd}
  \end{equation}
  In particular, it is a~$\bP^{g-1}$\dash bundle over~$\Pic^g(C)$.
\end{proposition}

\begin{proof}
  The image of~$\Pic^g(C)$ in~$\Gr(g,2g+2)$ consists of~$g$\dash dimensional subspaces
  which are isotropic for the quadratic forms~$q_1$ and~$q_2$ defined on the~$2g+2$\dash dimensional vector space~$V$.
  Hence we can describe the fiber product as
  \begin{equation}
    \label{equation:fiber-product-description}
    \{E_2\subset E_1\subset V\mid\dim E_2=g-2,\dim E_1=g,q_1(E_1)=q_2(E_1)=0\}.
  \end{equation}
  Recall that~$\fano(\OGr(g-1,2g+2;q_i))$ is again a homogeneous variety,
  namely the quotient~$\SO_{2g+2}/\rP_{g-2,g}$,
  where~$\rP_{g-2,g}$ denotes the parabolic subgroup
  \begin{equation}
%    \dynkin[labels={1,,g-2,,g,g+1}]{D}{**.x*x*}.
  \end{equation}
  Now observe that
  \begin{equation}
    \begin{aligned}
      \fano(\odd)&=\fano(\OGr(g-1,2g+2;q_1)\cap\OGr(g-1,2g+2;q_2)) \\
      &\subset\fano(\Gr(g-1,2g+2))\cong\Fl(g-2,g,2g+2)
    \end{aligned}
  \end{equation}
  allows us to describe the Fano variety of lines
  using the homogeneous spaces~$\SO_{2g+2}/\rP_{g-2,g}$,
  giving the description from~\eqref{equation:fiber-product-description}.
\end{proof}

From \cref{section:periods}
we recall that to compute the second quantum period,
we are interested in the moduli space~$\odd_{0,1,\beta}$
for~$\beta=2$ (because the Fano index of~$\odd$ is~2).
Again we wish to use the geometry of
the ambient partial flag varieties,
now using the isomorphism~$\Gr(g-1,2g+2)_{0,1,2}\cong\Fl(g-2,g-1,g,2g+2)$.
This allows us to obtain the following identification,
whose proof is similar to that of
\cref{proposition:fano-variety-interpretation}.
\begin{proposition}
  \label{proposition:012-interpretation}
  The variety~$\odd_{0,1,\beta}$ fits in the fiber product diagram
  \begin{equation}
    \begin{tikzcd}
      \odd_{0,1,2} \arrow[r, hook] \arrow[d, "\ev"] & \Fl(g-2,g-1,g,2g+2) \arrow[d] \\
      \Pic^g(C) \arrow[r, hook] & \Gr(g,2g+2).
    \end{tikzcd}
  \end{equation}
\end{proposition}

Summarizing, we have the following picture
\begin{equation}
  \begin{tikzcd}
    \odd_{0,1,2} \arrow[rd, dashed] \arrow[rr, hook, "f_1"] \arrow[rddd, crossing over] \arrow[dd, "\ev"] & & \Fl(g-2,g-1,g,2g+2) \arrow[dd] \arrow[rd, "f_2"] \arrow[rddd] \\
    & \cP \arrow[ld, near start, "f"] \arrow[dd, crossing over] \arrow[rr, hook, crossing over] & & \Fl(g-1,g,2g+2) \arrow[dd] \arrow[ld] \\
    \odd \arrow[rr, hook] & & \Gr(g-1,2g+2) \\
    & \Pic^g(C) \arrow[rr, hook] & & \Gr(g,2g+2).
  \end{tikzcd}
\end{equation}

To prove the vanishing of the second quantum period for~$g\geq 3$,
we need the following property of the evaluation morphism.
\begin{proposition}
The morphism~$\ev:\odd_{0,1,2}\to\odd$ factors through the natural map~$f:\cP\to\odd$ from \cref{lemma:P-interpretation}.
\end{proposition}

\begin{proof}
  It suffices to observe that the image of~$f_2\circ f_1$
  is contained in~$\cP\subset\Fl(g-1,g,2g+2)$,
  which can be proven along the lines of \cref{proposition:fano-variety-interpretation}.
\end{proof}

Hence we obtain the following corollary,
which is an independent proof of the vanishing result from \cref{proposition:p2},
without referring to the agreement of classical and quantum periods.
\begin{corollary}
  \label{corollary:p2-vanishing}
  Let~$g\geq 3$. Then~$p_2=0$.
\end{corollary}

\begin{proof}
  Because the image of~$\odd_{0,1,2}\to\odd$ under the evaluation morphism
  is contained in the image of~$\cP$,
  we have that the codimension of the image is at least~$g-2$.
  Hence the second quantum period vanishes for~$g\geq 3$, by the definition of the quantum period.
\end{proof}

\subsection{Fourth period}
The next interesting period is~$\pi_4$.
In the case we can give a closed formula for it.

\begin{proposition}
  \label{proposition:p4}
  The fourth coefficient of the period of the potential~$\widetilde{W}_{\graph,1}$ associated to~$\odd$ is given by
  \begin{equation}
    \pi_4(\widetilde{W}_{\graph,1})=
    \begin{cases}
      216 & g=2 \\
      192(g-1) & g\geq 3.
    \end{cases}
  \end{equation}
\end{proposition}

\begin{proof}
  The total degree of a monomial in the graph potential is either~3 or~$-1$,
  except for two monomials which are contributed by the (by assumption) unique colored vertex,
  whose total degrees are~1 and~$-3$.
  Let us moreover consider a graph potential associated to a graph without loops,
  so that all coefficients are~1.

  To understand the constant term of the fourth power of the graph potential,
  we consider the monomials of degree~$\pm3$.
  To obtain a constant term,
  we need to multiply this monomial with either a monomial of degree~$\mp3$,
  or with~3~monomials of degree~$\mp1$.
  If~$g=2$, then the colored vertex and uncolored vertex share the same variables,
  so it is possible to multiply monomials of degree~$\pm3$ and get a constant term.
  But this is impossible if~$g\geq 3$,
  and hence we assume from now on that we only multiply monomials of degree~$\pm3$ with monomials of degree~$\mp1$.

  There are~$2g-2$ monomials of degree~$\pm3$, one for each vertex.
  There are~4 ways to cancel this monomial:
  one way is by using the~3~distinct monomials contributed by the same vertex in the vertex potentials
  (see~\eqref{equation:vertex-potential} or \cite[Example~2.5]{gp-tqft}).
  The other three ways are by considering the edge potential for the~3~edges going out of the vertex
  This is an alternative way of writing the graph potential
  as a sum over the set~$E$ of edges instead of the set~$V$ of vertices,
  see \cite[\S3.3]{gp-decomp}.
  One can cancel the monomial by multiplying it with
  one monomial from the same vertex and two monomials from the neighbor.
  There are~$4!$ choices for doing this, hence the statement follows.
\end{proof}

%In fact, this proof also works in the uncolored case. But g=2 _and_ g=3 are a bit special now it seems, and g=3 being special isn't clear yet to me.

\subsection{General patterns}
\label{subsection:general-patterns}
The following result is completely general,
and immediate from the definition~\eqref{equation:descendants}.
Because~$\odd$ has index~2,
we have that~$\langle\beta,-\mathrm{K}_X\rangle$ is always divisible by~2.

\begin{lemma}
  \label{lemma:odd-vanishing}
  The odd-indexed coefficients of the quantum period of~$\odd$ vanish:
  \begin{equation}
    p_{2k+1}(\odd) = c_{2k+1}(\odd) = 0.
  \end{equation}
\end{lemma}

On the other hand, for the periods of graph potentials we have the following result.
It follows immediately from the observation that
the total degrees of the monomials in the graph potentials
are~$\pm1,\pm3$, so that no product of an odd number of them can have degree~0.

Let us write~$[W]_0$ for the constant coefficient of a Laurent polynomial~$W$.

\begin{lemma}
  \label{lemma:odd-vanishing-graph}
  For any odd or even coloring~$c$ of a graph $\graph$
  the odd-indexed coefficients of the period of the graph potential vanish:
  \begin{equation}
    [W_{\graph,c}(z)^{2k+1}]_{z^0} = 0.
  \end{equation}
\end{lemma}

This explains why in \cref{table:quantum-periods-odd,table:quantum-periods-even} at the end of this appendix
we have not listed the odd periods.
For the case of an odd coloring this is of course consistent with
\cref{proposition:equality-periods}
which gives the equality of the periods from
\cref{lemma:odd-vanishing,lemma:odd-vanishing-graph}.

We also have the following vanishing result.
\begin{lemma}
  \label{lemma:4k+2-vanishing}
  Let~$g\geq 2$.
  Then~$4\mid[W_{\graph,0}(z)^k]_{z^0}$.
\end{lemma}

\begin{proof}
  All monomials of the even graph potential $W := W_{\graph,0}$ have total degree~3 and~$-1$,
  so $W^k$ is homogeneous of degree $-k\in\mathbb{Z}/4\mathbb{Z}$.
  On the other hand, the constant term $z^0$ is homogeneous of degree $0\in\mathbb{Z}/4\mathbb{Z}$.
  % In computing the constant term of $W^k$ we are
  % choosing~$a$ monomials of positive degree,
  % and~$b$ monomials of negative degree,
  % such that~$k = a+b$
  % The product needs to be constant, so~$3a-b=0$.
  % But there are no integer solutions to these equations.
\end{proof}

Finally we have the following general result, explaining why the lower triangular parts of the tables agree.

\begin{proposition}
  Let~$g\geq 2$ and~$k<2g-2$.
  Then the~$k$th coefficient of the period of the graph potential for the even coloring
  equals that of the odd coloring.
\end{proposition}

Observe that if~$k$ is odd then we already knew that the~$k$th period is zero in both cases.

\begin{proof}
  By \cite[Corollary~2.18]{gp-tqft}
  we can consider a single graph~$\graph$ of genus~$g$,
  with an even (resp.~odd) coloring.
  This gives the graph potentials~$W_0$ (resp.~$W_1$).
  Moreover by \cite[Corollary~2.9]{gp-tqft} we have that the action of~$\CC_1(\graph,\mathbb{Z}/2\mathbb{Z})$
  preserves the graph potential up to biregular automorphism of the torus.

  By assumption we have that~$k$ is strictly less than the number of vertices in the graph~$\graph$.
  If~$m_1\cdots m_k$ is a product of monomials
  (which can also be considered as a sum of vectors in~$N_\graph$)
  contributing to the constant term of~$W_1^k$,
  taken from the product expansion of~$W_1^k$,
  it cannot use monomials from every vertex of the graph~$\graph$.

  Let~$v$ be an unused vertex.
  By the action of~$\CC_1(\graph,\mathbb{Z}/2\mathbb{Z})$
  we can assume that the colored vertex is~$v$,
  without changing the value of~$m_1\cdots m_k$,
  as by assumption this is a constant
  and hence invariant under biregular automorphisms of the torus.
  Therefore this is the same contribution as obtained from~$W_0^k$, and we are done.
\end{proof}

Combining this with \cref{lemma:4k+2-vanishing} we get
\begin{corollary}
  Let~$k\geq 0$. Let~$g\geq 2$. If~$4k+2<2g-2$ then the~$4k+2$th period of the graph potential in the odd case vanishes.
\end{corollary}

\include{sympl-table}

%% file: sympl-table.tex
%\begin{landscape}
\begin{sidewaystable}
  \vspace{10cm}
  \centering
  \begin{tabular}{cccccccccc}
    \toprule
    $g$ & $p_0$ & $p_2$ & $p_4$ & $p_6$ & $p_8$     & $p_{10}$   & $p_{12}$       & $p_{14}$        & $p_{16}$             \\
    \midrule
    2   & 1     & 8     & 216   & 8000  & 343000    & 16003008   & 788889024      & 40424237568     & 2131746903000        \\
    3   & 1     & 0     & 384   & 23040 & 3265920   & 435456000  & 68263641600    & 11300889600000  & 1984905402480000     \\
    4   & 1     & 0     & 576   & 11520 & 8769600   & 1175731200 & 445839609600   & 115772770713600 & 41211916193448000    \\
    5   & 1     & 0     & 768   & 0     & 16853760  & 928972800  & 1378578432000  & 295708763750400 & 237075779068128000   \\
    6   & 1     & 0     & 960   & 0     & 27518400  & 232243200  & 3112327680000  & 299893321728000 & 795162277629720000   \\
    7   & 1     & 0     & 1152  & 0     & 40763520  & 0          & 5892216422400  & 133905855283200 & 2006716647119184000  \\
    8   & 1     & 0     & 1344  & 0     & 56589120  & 0          & 9963493478400  & 22317642547200  & 4248683870158728000  \\
    9   & 1     & 0     & 1536  & 0     & 74995200  & 0          & 15571407667200 & 0               & 7983708676751808000  \\
    10  & 1     & 0     & 1728  & 0     & 95981760  & 0          & 22961207808000 & 0               & 13760135544283128000 \\
    \bottomrule
  \end{tabular}
  \caption{Period sequence for the odd graph potential \\ Quantum periods for $\moduli_C(2,\mathcal{L})$ \\ (see also \cite[Table~1]{gp-tqft})}
  \label{table:quantum-periods-odd}

  \vspace{1cm}

  \centering
  \begin{tabular}{ccccccccccc}
    \toprule
    $g$ & $p_0$ & $p_2$ & $p_4$ & $p_6$ & $p_8$     & $p_{10}$   & $p_{12}$       & $p_{14}$        & $p_{16}$             & $p_{18}$ \\
    \midrule
    2   & 1     & 0     & 384   & 0     & 645120    & 0          & 1513881600     & 0               & 4132896768000        & 0 \\
    3   & 1     & 0     & 576   & 0     & 6350400   & 0          & 136604160000   & 0               & 3976941969000000     & 0 \\
    4   & 1     & 0     & 576   & 0     & 12640320  & 0          & 805929062400   & 0               & 80306439693480000    & 0 \\
    5   & 1     & 0     & 768   & 0     & 18144000  & 0          & 1915060224000  & 0               & 401643111149280000   & 0 \\
    6   & 1     & 0     & 960   & 0     & 27518400  & 0          & 3418888704000  & 0               & 1062973988196120000  & 0 \\
    7   & 1     & 0     & 1152  & 0     & 40763520  & 0          & 5953528627200  & 0               & 2211592605702480000  & 0 \\
    8   & 1     & 0     & 1344  & 0     & 56589120  & 0          & 9963493478400  & 0               & 4323671149117320000  & 0 \\
    9   & 1     & 0     & 1536  & 0     & 74995200  & 0          & 15571407667200 & 0               & 7994421145174464000  & 0 \\
    10  & 1     & 0     & 1728  & 0     & 95981760  & 0          & 22961207808000 & 0               & 13760135544283128000 & 0 \\
    \bottomrule
  \end{tabular}
  \caption{Period sequence for the even graph potential \\ (see also \cite[Table~2]{gp-tqft})}
  \label{table:quantum-periods-even}
\end{sidewaystable}
%\end{landscape}